\documentclass[12pt,reqno]{article}
\usepackage[latin1]{inputenc}
\usepackage{amsxtra}
\usepackage{amsmath}
\usepackage{amssymb}
\usepackage{amsfonts}
\usepackage[all]{xy}
\usepackage{mathrsfs}
\usepackage{amsthm}
\usepackage{enumerate}
\usepackage{epigraph}

\newtheorem{thm}[equation]{Theorem}
\newtheorem{lemma}[equation]{Lemma}
\newtheorem{prop}[equation]{Proposition}
\newtheorem{cor}[equation]{Corollary}

\theoremstyle{definition}
\newtheorem{df}[equation]{Definition}

\newtheoremstyle{example}{\topsep}{\topsep}%
     {}
     {}
     {\bfseries}
     {.}
     {2pt}
     {\thmname{#1}\thmnumber{ #2}\thmnote{ #3}}

   \theoremstyle{example}
\newtheorem{rem}[equation]{Remark}

\newtheorem{ex}[equation]{Example}

\numberwithin{equation}{section}

\def\a{{\alpha}}

\def\CC{\mathbb{C}}
\def\FF{\mathbb{F}}

\def\PP{\mathbb{P}}
\def\RR{\mathbb{R}}
\def\ZZ{\mathbb{Z}}
\def\QQ{\mathbb{Q}}
\def\HH{\mathbb{H}}

\def\men{\mathfrak{m}}

\def\Aen{\mathfrak{A}}
\def\Sen{\mathfrak{S}}
\def\Een{\mathfrak{E}}

\def\Ac{\mathcal{A}}
\def\Bc{\mathcal{B}}
\def\Cc{\mathcal{C}}

\def\Dc{\mathcal{D}}

\def\Fc{\mathcal{F}}

\def\Mc{\mathcal{M}}
\def\Nc{\mathcal{N}}
\def\Hc{\mathcal{H}}
\def\Oc{\mathcal{O}}

\def\Rc{\mathcal{R}}
\def\Sc{\mathcal{S}}

\def\Wc{\mathcal{W}}

\def\Hom{{\operatorname{Hom}}}

\def\dim{{\rm{dim}}}

\def\Shuff{{\Sc \Hc}}
\def\SShuff{{\Sc\Sc\Hc}}

\def\Sen{\mathfrak{S}}

\def\Bb{\mathbf{B}}
\def\Eb{\mathbf{E}}

\def\RRe{\operatorname{Re}}

\def\homo{\operatorname{\it \mathscr{H}\kern-.25em om}}
\def\ext{\operatorname{\it \mathscr{E}\kern-.25em xt}}
\def\edo{\operatorname{\it \mathscr{E}\kern-.25em nd}}
\def\der{\operatorname{\it \mathscr{D}\kern-.25em er}}

\def\Coh{{\mathcal{C}oh}}

\def\Bun{{\mathcal{B}un}}
\def\Bunn{\operatorname{Bun}}

\def\Hom{\operatorname{Hom}\nolimits}

\def\diag{{\operatorname{diag}\nolimits}}

\def\supp{{\operatorname{supp}\nolimits}}

\def\Spec{\operatorname{Spec}\nolimits}

\def\Ext{\operatorname{Ext}\nolimits}

\def\Ker{\operatorname{Ker}\nolimits}

\def\deg{{\operatorname{deg}\nolimits}}

\def\Pic{{\operatorname{Pic}\nolimits}}
\def\rk{\operatorname{rk}\nolimits}

\def\Im{\operatorname{Im}\nolimits}

\def\st{{\operatorname{st}\nolimits}}

\def\Id{\operatorname{Id}\nolimits}

\def\tr{\operatorname{tr}\nolimits}

\def\Sym{\operatorname{Sym}\nolimits}

\def\Supp{\operatorname{Supp}\nolimits}

\def\SZ{\overline{\Spec (\ZZ)}\nolimits}
\def\Vol{\operatorname{Vol}\nolimits}
\def\res{\operatorname{res}\nolimits}

\def\SZ{\overline{\Spec(\ZZ)}}
\def\PW{\mathcal{PW}}
\def\Mer{\Mc er}
\def\CT{\operatorname{CT}}
\def\Temp{\operatorname{Temp}}
\def\bbs{{\backslash\hskip -1mm \backslash}}
\def\Dist{{\Dc ist}}

\def\3x3{\operatorname{3}\times\operatorname{3}}

\def\1{{\bf 1}}
\def\lra{\longrightarrow}
\def\lla{\longleftarrow}
\def\(({(\hskip -1mm (}
\def\)){)\hskip -1mm )}
\def\be{\begin{equation}}
\def\ee{\end{equation}}

\def\sm{\circledS}
\def\on{\operatorname}
\def\Lamed{\mathfrak{L}}
 \def\dep{\on{dpt}}
 \def\Vertt{\on{Vert}}

\title{The spherical Hall algebra of $\SZ$}

\author{M. Kapranov, O. Schiffmann, E. Vasserot}

\begin{document}


\maketitle

\thanks{\em To Yuri Ivanovich Manin on his 75th birthday}

\vskip 1cm

\tableofcontents

 \addtocounter{section}{-1}
 
 \vfill\eject

\section{Introduction.}



\paragraph{(0.1)}
The construction of the Hall algebra of an abelian category $\Ac$  is known to
produce interesting Hopf algebras of quantum group-theoretic nature.
A condition usually imposed to ensure that 
  the Hall algebra has a compatible comultiplication, is that $\Ac$
 is hereditary  (of homological dimension 1). 
 There are two main types of hereditary abelian categories which have
 been studied in this respect. 
 
 \vskip .2cm
 
 First, if $Q$ is a {\em quiver}, we can form the category $\Ac = \Rc ep_{\FF_q}(Q)$
 of (finite-dimensional) representations of $Q$ over a finite field $\FF_q$.
  As discovered by Ringel \cite{ringel}, the   
 Hall algebra of $\Rc ep_{\FF_q}(Q)$
  is related to the quantized Kac-Moody algebra  whose Dynkin diagram is $Q$.
  More precisely, it contains $U_q(\mathfrak{n}_+)$,
   the quantization of the unipotent subalgebra on the positive
  root generators from the Kac-Moody root system. 
 
 \vskip .2cm
 
 Second, if $X$ is a smooth projective {\em curve} over $\FF_q$, we can form the
 category $\Ac=\Cc oh(X)$ of coherent sheaves on $X$. In this case the
 Hall algebra contains the spaces of unramified automorphic forms on the 
 groups $GL_r$, $r\geqslant 1$
 over the function field $K=\FF_q(X)$, and the multiplication correponds to
 forming  Eisenstein series  \cite{kapranov}. One can  also
  include  ``orbifold curves" $G\bbs X$ where $G$ is a finite group of
 automorphisms of a curve $X$, see \cite{schiffmann-orbifold}. 
 The algebras obtained in this way
 include both quantum affine algebras 
 \cite{kapranov, schiffmann-orbifold} and  
   spherical Cherednik algebras 
   \cite{schiffman-vasserot:elliptic}.

\paragraph {(0.2)} The goal of the present paper is to begin the study of a third, more arithmetic, type
of Hall algebras. It is obtained by replacing a curve $X/\FF_q$ by the spectrum
of the {\em ring of integers in a number field}, compactified at infinity 
by the Archimedean valuations. 
In this paper we consider only 
the basic example  of
$\SZ = \operatorname{Spec}(\ZZ)\cup \{\infty\}$. 
 The  role of   rank $n$ vector bundles for $\SZ$
is played by  free  abelian groups  $L$  of rank $n$ with a positive definite quadratic
form in $L\otimes\RR$, see
 \cite{stuhler1, stuhler2, grayson} as well as 
 \cite{manin, soule} for a more general point of view of Arakelov geometry. 
 The ``moduli space" of such bundles
 is the classical quotient
of reduction theory of quadratic forms
$$\Bunn_n \,\,=\,\, GL_n(\ZZ)\backslash GL_n(\RR)/O_n.$$
  Functions on $\Bunn_n$ are the same as 
 automorphic forms on $GL_n(\RR)$,
see \cite{goldfeld} for a detailed study of precisely this
situation. 
 
 \vskip .2cm

\paragraph{(0.3)} To describe our arithmetic analog of the Hall algebra, 
let $H_n=C^\infty_0(\Bunn_n)$ be the space of smooth functions
on $\Bunn_n$ with compact support.   The space
 $H=\bigoplus_n H_n$ has a natural structure of  an associative algebra,
 constructed in \S \ref{sec:hall-alg}. From the point of
view of the automorphic form theory, the multiplication in $H$ is
given by the parabolic pseudo-Eisenstein series map. 
If $X$ is a curve over $\FF_q$,
the analogous map for unramified automorphic forms over the function field $\FF_q(X)$
gives the multiplication in the Hall algebra of $X$, see \cite{kapranov}. 
So in this paper we study the space $H$ of automorphic forms on all the $GL_n(\RR)$
 as an associative algebra in its own right. 

\vskip .2cm

We further concentrate on
the subalgebra $SH\subset H$ generated by $H_1=C^\infty_0(\RR_{>0})$. 
Extending the terminology of \cite{schiffmann-vasserot:higher-genus},
 we call $SH$ the {\em spherical Hall algebra} of $\SZ$. 
 From the point of view of spectral decomposition
 \cite{moeglin-waldspurger-book}, $SH$ consists of automorphic forms
 expressible through the Eisenstein-Selberg series \cite{selberg},
 the simplest higher-dimensional analogs of the nonholomorphic Eisenstein-Mass
 series on the upper half plane. 
 This algebra has an explicit space of generators, but relations
 among these generators are not directly given.

\paragraph{(0.4)}
Our first main result describes $SH$ as a Feigin-Odesskii-type shuffle algebra,
in a way similar to the results of 
\cite{schiffmann-vasserot:higher-genus} for the case of curves over a finite field.
  However, in our case the
shuffle algebra is based not on a rational, but on a meromorphic
function: the Riemann zeta function $\zeta(s)$. This function, therefore, encodes
all the relations among the generators from $H_1$. 

\vskip .2cm

 Quadratic relations in $SH$ correpond to the classical functional equation for the
Eisenstein-Maass series,  in a way similar to the case of function field considered in \cite{kapranov}. 
One form of writing the relations is in terms of ``generating functions"
 (formal $H$-valued distributions) $\Een(s)$
depending on $s\in\CC$. It has the form
\[
\Een(s_1) \Een(s_2) \,\,=\,\, {\zeta^*(s_1-s_2)\over\zeta^*(s_1-s_2+1)} \Een(s_2)\Een(s_1),
\]
where $\zeta^*(s)$ is the full zeta function of $\SZ$
 (the product of $\zeta(s)$ with   the Gamma  and exponential factors).
This is discussed in \S \ref{sec:quadration-relations}.

\vskip .2cm

Our second main result, Theorem
\ref{thm:cubic-relations-zeta},
is that the space of the cubic relations (not following from the
quadratic ones) is identified with (an appropriate completion of) 
the space   spanned by nontrivial zeroes of $\zeta(s)$. 
In other words, {\em the space spanned by the zeroes of $\zeta(s)$ can be
realized as a
  certain algebraic homology space of  the associative algebra $H$.} 
  This is remindful of (but different from) the result of D. Zagier \cite{zagier}
  who gave an interpretation of the zeta-space using integrals of Eisenstein-Maass
  series over anisotropic tori associated to real quadratic fields.

 \paragraph{(0.5)} After the first draft of this paper was written, we learned that M. Kontsevich
 and Y. Soibelman \cite{KS} have recently considered the algebra $H$ as well.  Their 
 interest was in studying  wall-crossing formulas in $\Bun$, so our results practically do not intersect.
 We are grateful to M. Kontsevich
 and Y. Soibelman for explaining their work and providing us with the
 preliminary version of \cite{KS}.

  \paragraph{(0.6)}  M.K.  would like to thank Universities Paris-7 and Paris-13 
  as well as the Max-Planck Institut f\"ur Mathematik in Bonn for hospitality and support during the
  work on this paper. His research was also partially supported by an NSF grant.

\vskip 1cm

\section{Vector bundles on $\overline{\Spec(\ZZ)}$.}
\label{sec:bundles-SZ}
 
 By a {\em vector bundle on} $\SZ$ we will mean a triple
$E=(L,V,q)$, where $V$ is a finite-dimensional $\RR$-vector
space, $q$ is a positive definite quadratic form on $V$,   and $L\subset V$
is a $\ZZ$-lattice of maximal rank. In this case, $V$ becomes a Banach
space with  norm $\|v\|=\sqrt{q(v)}$.

\vskip .2cm

The {\em rank} of $E$
is defined as $\rk(E)=\dim_\RR(V) = \rk_\ZZ(L)$.
 A {\em morphism} 
$
f: E'=(L',V',q')\lra E=(L, V, q)
$
of vector bundles
on $\SZ$ is, by definition, a linear operator $f: V'\to V$ such that, first,
$f(L')\subset L$ and, second, $\|f\|\leqslant 1$, i.e., we have $q(f(v'))\leqslant q'(v')$
for each $v'\in V'$. In this way we get a category which we denote 
  $\Bun$. All the Hom-sets in $\Bun$ are finite. 
  
  We denote by $\Oc=(\ZZ, \RR, x^2)$ the {\em trivial bundle of rank 1}.
  
  \vskip .2cm

 The {\em dual bundle} to $E$ is defined as $E^\vee=(L^\vee, V^*, q^{-1})$, where
 $q^{-1}$ is the inverse quadratic form on the dual space. The {\em tensor product}
 of two bundles is defined as
 $$E\otimes E' = (L\otimes_\ZZ L', V\otimes_\RR V', q\otimes q'), \quad (q\otimes q')(v\otimes v'):= q(v)q'(v').$$
  In particular, we have the bundle $\underline{\Hom}(E, E')=E^\vee\otimes E'$. The corresponding
 quadratic form on $\Hom_\RR(V, V')$ takes $f: V\to V'$ into $\tr(f^t\circ f)$, where the transpose is taken with
respect to $q,q'$. We leave to the reader the proof of the following:
 
 \begin{prop} Let $E_i=(L_i, V_i, q_i)$, $i=1,2,3$, be three vector bundles on $\SZ$. Then
 $$\Hom_{\Bun}(E_1, \underline{\Hom}(E_2, E_3)) \,\,\subset\,\,\Hom_{\Bun}(E_1\otimes E_2, E_3)$$
 as subsets in $\Hom_\RR(V_1\otimes V_2, V_3)$. \qed
  \end{prop}
  
  Note the particular case of $E_1=\Oc$. The proposition in this case reduces to the inequality
  $$\|f\| \,\leqslant \sqrt{\tr(f^t\circ f)}$$
  for any linear operator $f: V_2\to V_3$. We also see why the inclusion in the proposition
  is not, in general, an equality. Indeed, for $E_1=\Oc$, the Hom-set on the left consists of integer points
  in the domain $\tr(f^t\circ f)\leqslant 1$, which is an ellipsoid. But the Hom-set on the right
  consists of integer points in the domain $\|f\|\leqslant 1$ which is not an ellipsoid, if $\dim(V_2),\dim(V_3) > 1$. 
  
 \vskip .2cm

 We also have the symmetric and exterior product functors
 $$\begin{gathered} S^r(E) = (S^r_\ZZ(L), S^r_\RR(V), S^r(q)),
  \quad S^r(q)(v_1\bullet \cdots \bullet v_r) :=
  q(v_1) \cdots q(v_r),
 \cr
 \Lambda^r(E) = (\Lambda^r_\ZZ(L), \Lambda^r_\RR(V), \Lambda^r(q)), \quad 
 \Lambda^r(q)(v_1\wedge \cdots \wedge v_r)  := \det\| B(v_i, v_j)\|.
 \end{gathered}
 $$
 Here $\bullet$ is the product in the symmetric algebra, while
$B$ is the symmetric bilinear form such that $q(v)=B(v,v)$.

\vskip .2cm

Let $\Bunn_n$ be the set of isomorphism classes of rank $n$ vector bundles on $\SZ$.
This set is  a classical double quotient of the theory of automorphic forms:
\be\label{eq:bijection-Z} 
\Bunn_n \buildrel \sim\over\lla GL_n(\ZZ)\backslash GL_n(\RR)/O_n.
\ee
Explicitly, the double coset of $g_\infty\in GL_n(\RR)$ corresponds to the
isomorphism class of the bundle
$(\ZZ^n, \RR^n,  (g_\infty^t)^{-1}_*(q_\st))$, where 
  \[
  q_\st(x_1, ..., x_n)\,\, =\,\,\sum_{i=1}^n  x_i^2
  \]
   is the standard quadratic form on $\RR^n$
  and $(g_\infty^t)^{-1}_*(q_\st)(x)=q_\st((g_\infty^t)^{-1}(x))$ is the quadratic form corresponding to the
   symmetric matrix
  $(g_\infty^t)^{-1}\cdot g_\infty^{-1}$.  
  
  \vskip .2cm

We will also need an   adelic version of \eqref{eq:bijection-Z}.
Let $\Aen^f =  \prod_p^{\res}\QQ_p$ be the ring of  finite adeles of the field $\QQ$, 
let $\widehat \ZZ = \prod_p \ZZ_p\subset\Aen^f$ be the profinite completion of $\ZZ$,
and $\Aen = \RR\times \Aen^f$ be the full ring of adeles. Then 
  $K_n := O_n \times\prod_p GL_n(\ZZ_p)$  is a maximal compact subgroup of
$GL_n(\Aen)$.

\begin{prop}\label{prop:bun-matrices} The embedding of $GL_n(\RR)$ into $GL_n(\Aen)$
induces a bijection 
 $$\Bunn_n \simeq  GL_n(\ZZ)\backslash GL_n(\RR)/O_n\buildrel \a\over\lra
GL_n(\QQ)\backslash GL_n(\Aen)/K_n. $$

\end{prop}

\begin{proof} The statement is of course well known. 
 We describe the inverse map explicitly for later use.
 Let $g=(g_\infty, (g_p))\in GL_n(\Aen)$, so $g_\infty\in GL_n(\RR)$ and
 $g_p\in GL_n(\QQ_p)$, with $g_p\in GL_n(\ZZ_p)$ for almost all $p$.
 We associate to $g$ a vector bundle $E_g=(L_g, V_g, q_g)$
 on $\SZ$ by putting:
 \[
 L_g \,\,=\,\, \QQ^n \,\cap\,\bigcap_p g_p^t(\ZZ_p^n), \quad
 V_g=\RR^n, \quad
 q_g =  (g_\infty^t)^{-1}_*(q_\st). 
 \]
 It is clear that $E_{\gamma g k}\simeq E_g$ for $\gamma\in GL_n(\QQ)$,
 $k\in K_n$, so we get a map
 \[
 GL_n(\QQ)\backslash GL_n(\Aen)/K_n\buildrel \beta\over\lra \Bunn_n.
 \]
 By construction, $\beta \a=\Id$; the fact that $\a\beta=\Id$ follows since $L_g$ is a free
 abelian group. \end{proof}

\begin{ex}\label{ex:picard}
 Take $n=1$. The set $\Bunn_1$ formed by isomorphism classes
of line bundles, will be also denoted by $\Pic(\SZ)$. This set
is a group under tensor multiplication. It is identified
with $\RR_{+}^\times$, the multiplicative group of positive real numbers.
Explicitly, given $E=(L,V,q)$ with $\dim_\RR(V)=1$, we associate to it
the number $\deg(E) = 1/\sqrt{q(l_{\min})}\in \RR_{+}$, where $l_{\min}$
is one of the two generators of $L$. Conversely, for $a\in \RR_{+}$
we denote by
$\Oc(a)\,\,=\,\,(\ZZ, \RR, a^{-2}\cdot q_\st)$
the corresponding line bundle with $\deg(\Oc(a))=a$. The convention,
compatible with \eqref{eq:bijection-Z}  for $n=1$, is
chosen so that for $a\gg 0$ the bundle $\Oc(a)$ has many "global sections",
i.e., lattice points $l$ such that $q(l)\leqslant 1$.
\end{ex}

\begin{ex}\label{ex:Bun-2-1}
 For any $n$, taking the top exterior power together with the
isomorphism of Example \ref{ex:picard}, gives a map
$$\Bunn_n\buildrel\det\over\lra\Pic(\SZ) \buildrel \sim\over\lra \RR_{+}.$$
Explicitly, $E=(L,V,q)$ is sent into $1/\Vol(V/L)$, the inverse of the covolume
of $L$ with respect to the Lebesgue measure defined by $q$. 
We will denote this inverse covolume by $\deg(E)$ and call it the degree of $E$. 
We denote by $\Bunn_{n,a}$ the set of isomorphism classes of
bundles of rank $n$ and degree $a$. 

\vskip .2cm

Consider the case $n=2$ and take $a=1$. In this case
$$\Bunn_{2,1}\,\,=\,\, SL_2(\ZZ)\backslash SL_2(\RR)/SO_2$$
is identified with the quotient $SL_2(\ZZ)\backslash\HH$, where $\HH\subset\CC$
is the upper half-plane $\Im(z) >0$. 
More explicitly, consider the standard quadratic form on $\CC$ given by
$q_\st(z) = |z|^2$. Then, for $\tau\in\HH$,  the lattice $\ZZ + \ZZ\tau$ has, with respect to
$q_\st$, the covolume equal to $\Im(\tau)$. We therefore associate to $\tau$
the bundle
$$E_\tau\,\,=\,\, \bigl( \ZZ + \ZZ\tau, \,\CC, \,q_\st/\Im(\tau)^{1/2}\bigr)\,\,\in\,\,\Bunn_{2,1}.$$
\begin{lemma} For $\gamma\in SL_2(\ZZ)$ we have $E_{\gamma(\tau)} \simeq E_\tau$, and this
establishes an identification $SL_2(\ZZ)\backslash H\to \Bunn_{2,1}$. 
\end{lemma}

\begin{proof} It is clear that $E_\tau\simeq E_{\tau +1}$. Let us show that
$E_{-1/\tau}\simeq E_\tau$. Note that
$$\Vol(\CC/L_{-1/\tau}) \,\,=\,\,\Im(-1/\tau)\,\, =\,\,
\Im\biggl({-\overline\tau\over |\tau|^2}\biggr).
$$
Notice also that  multiplication by $\tau$ defines an isomorphism of lattices
\[
L_{-1/\tau}\buildrel \tau\over\lra L_\tau.
\]
The determinant of the multiplication
by $\tau$ being $|\tau|^2$, we conclude that this multiplication defines an isomorphism
$$ \bigl( L_{-1/\tau}, \,\CC,\, q_{st}/\Im(-1/\tau)^{1/2}\bigr)
\lra
\bigl(L_\tau,\,\CC,\, q_\st/\Im(\tau)^{1/2}\bigr)$$
of vector bundles over $\SZ$.  \end{proof}

\end{ex}

Let now
\be\label{eq:SES}
0\to E'=(L',V',q')\buildrel i\over\lra E=(L,V,q)\buildrel j\over\lra E''=(L'',V'',q'')\to 0
\ee
be a sequence  of vector bundles on $\SZ$ and their morphisms. 

\begin{df}\label{def:SES-Bun} 
We say that a sequence \eqref{eq:SES} is  {\em short exact} (in $\Bun$),
if the following hold:

\begin{itemize}
\item[(1)] The induced sequences of vector spaces and abelian groups are short exact.

\item[(2)] The form $q'$ is equal to $i^*(q)$, the pullback of $q$ via $i$, defined by
 \[ (i^*q)(v')=q(i(v')),  \quad  v'\in V'.
 \]

\item[(3)] The form $q''$ is equal to $j_*(q)$, the pushforward of $q$ via $j$, defined by
\[
(j_*q)(v'') = \min\limits_{j(v)=v''} q(v), \quad  v''\in V''.
\]  
\end{itemize}

\noindent An {\em admissible monomorphism} (resp.
{\em  admissible epimorphism}) in $\Bun$
is a morphism which can be included into a short exact sequence as
$i$ (resp. $j$). 
\end{df}

Let us call a {\em subbundle} in $E$ an equivalence class of admissible
monomorphisms $E'\to E$ modulo isomorphisms of the source.
For such a subbundle $E'$ we have the quotient bundle $E/E'\in\Bun$. 

\begin{prop} Let $E=(L,V,q)$ be a vector bundle on $\SZ$. The following sets are
in bijection:

(i) Rank $r$ subbundles $E'\subset E$.

(ii) Rank $r$ primitive sublattices, i.e., subgroups $L'\subset L$ such that $L/L'$ has no
torsion.

(iii) $\QQ$-linear subspaces $W'\subset L\otimes_\ZZ\QQ$ of dimension $r$. 
\end{prop}

\begin{proof} The bijection between (ii) and (i) takes a primitive sublattice $L'$
into $E'=(L', V',q')$, where $V'=L'\otimes_\ZZ\RR$ and $q'=q|_{V'}$. The bijection
between (iii) and (ii) takes a subspace $W'$ into the sublattice $L'=L\cap W'$.
\end{proof}

\begin{prop}\label{prop:finiteness-subbundles} Let $E=(L,V,q)$ be a vector bundle
on $\SZ$. 
 For any $r\in\ZZ_+$ and $a\in \RR_{+}$, the set of subbundles $E'\subset E$
with $\rk(E')=r$ and $\deg(E')\geqslant a$, is finite.
\end{prop}

\begin{proof} Let $W=L\otimes_\ZZ\QQ$. 
 Consider first the case $r=1$.  If $E'\subset E$ corresponds
to a 1-dimensional subspace $W'\subset W$, then $\deg(E) =1/\sqrt{ q(w')}$, where $w'\in W'\cap L=L'$
is one of two generators of this free abelian group of rank $1$. Since the number of $w'$
such that $q(w')\leqslant a$ is finite, our statement follows.

Consider now the case of arbitrary $r$ and use the Pl\"ucker embedding of the Grassmannian
$G(r,W)$  into $\PP(\Lambda^r(W))$. If $W'\subset W$ is an $r$-dimensional subspace
with $L'=W'\cap L$, then $\Lambda^r(W')\subset\Lambda^r(W)$ is a 1-dimensional
subspace, and
$\Lambda^r(W')\cap \Lambda^r_\ZZ(L) = \Lambda^r_\ZZ(L')$ is a free abelian group of rank 1
and a primitive sublattice in $\Lambda^r_\ZZ(L)$. Further, $\Lambda^r_\RR(V)$ is
equipped with the quadratic form $\Lambda^r(q)$, and $\deg(E')=1/\sqrt{\Lambda^r(q)(\xi')}$,
where $E'$ is the subbundle corresponding to $W'$ and $\xi'\in\Lambda^r_\ZZ(L')$ is one of
the two generators. We thus reduce to the case of subbundles of rank 1. 
\end{proof}

\vskip .2cm

Let $E'=(L',V',q')$ and , $E''=(L'',V'',q'')$ be two vector bundles on $\SZ$. We define $\Ext^1(E'', E')$
to be the set of admissible short exact sequences \eqref {eq:SES}
 modulo automorphisms of such sequences identical on $E'$ and $E''$.

\begin{prop}\label{prop:Ext-E''-E'}
The set $\Ext^1(E'', E')$ has a natural structure of a $C^\infty$-manifold
isomorphic to the torus $(\RR/\ZZ)^{n'n''}$, where $n'=\rk(E')$ and $n''=\rk(E'')$.
\end{prop}

\begin{proof}  For any short exact sequence
as in \eqref{eq:SES}, the induced short exact sequence of lattices
  necessarily splits. Let us fix a splitting $L=L'\oplus L''$ and the
induced splitting $V=V'\oplus V''$ of $\RR$-vector spaces, so that $i$ and $j$ become
the canonical embedding into and the projection from the direct sum. Let $\Lambda$
be the set of positive definite quadratic forms $q$ on $V$ such that
$i^*q=q'$ and $j_*q=q''$. By definition, $\Lambda$ is a closed subset in the
space of all positive definite quadratic forms on $V$ and so has a natural
topology.

The group $\Hom_\ZZ(L'', L')$ is identified with the group of automorphisms of the
split exact sequence
$$0\to L'\buildrel i\over\lra L'\oplus L''\buildrel j\over\lra L''\to 0$$
identical on $L', L''$. Therefore this group acts on $\Lambda$, and we have
$\Ext^1(E'', E')\,\,=\,\,\Lambda/\Hom_\ZZ(L'', L')$.

\begin{lemma} The map $\Lambda\buildrel \res\over\lra \Hom_\RR(V'\otimes V'',\RR)$
which sends $q$ into the induced pairing between the summands $V'$ and $V''$,
is a homeomorphism. This map takes the action of the group $\Hom_\ZZ(L'',L')$
on
$\Lambda$ into its action on  $\Hom_\RR(V'\otimes V'', \RR)$ by translations.
\end{lemma}

\noindent {\sl Proof:}
Fix a basis $e_1, ..., e_{n'}$ of $V'$, orthonormal with respect to $q$,
and a basis $v_1, ..., v_{n''}$ of $V''$. Let $B', B''$ be the
symmetric bilinear forms on $V', V''$ corresponding to $q', q''$, and let
$q$ be a quadratic form on $V$ with corresponding symmetric bilinear form $B$.
Then the condition $q\in\Lambda$ means:
\be\label{eq:q-in-Lambda}
\begin{gathered}
B(e_i, e_j) \,=\,\delta_{ij}\,=\,B'(e_j, e_j),\cr
B\biggl(v_p-\sum_{\mu =1}^{n'} B(v_p,e_\mu)\cdot e_\mu, \,\,
v_q-\sum_{\nu =1}^{n'} B(v_q,e_\nu)\cdot e_\nu\biggr)\,\,=\,\,B''(v_p, v_q).
\end{gathered}
\ee
Indeed, the  
minimum in the definition of $j_*q$ is given by the
orthogonal projection to $V'$ with respect to $B$,
and the left hand side of the second formula above involves
exactly such projections.

Denote by $X$ the matrix $\|B(v_p, e_\mu)\|$ of size $n''\times n'$, and let
$Y$ be the matrix $\|B(v_p, v_q)\|$ of size $n''\times n''$. From the first
condition in \eqref{eq:q-in-Lambda} we see that a quadratic form $q$
with $i^*q=q'$ is completely
determined by the datum of $X$ and $Y$, while the second equation implies that
$Y=B''-X\cdot X^t$, where $B''=\|B''(v_p, v_q)\|$. Therefore $q\in\Lambda$
is indeed completely defined by $X$, which is the matrix representative
of $\res(q)$. The action of elements of $\Hom_\ZZ(L'', L')$
in the matrices $X$ is the action by translation. This
proves the lemma and Proposition \ref{prop:Ext-E''-E'}. 
\end{proof}

\begin{rem}\label{rem:coherent-sheaves}
More generally, one can consider data $\Fc = (L,V,q)$ similar to the above but where $L$
is any finitely generated abelian group, $V=L\otimes_\ZZ \RR$ and $q$ is a positive definite quadratic
form on $V$. They correspond   to {\em coherent sheaves on $\SZ$
 locally free at infinity}. We get in this way a category $\Coh_{\neq\infty}(\SZ)$, with admissible short
 exact sequences defined similarly to Definition \ref{def:SES-Bun}. A more systematic
 theory should enlarge  $\Coh_{\neq\infty}(\SZ)$ by allowing a meaningful concept of sheaves with
  torsion at $\infty$. This will be done in a subsequent paper. For example,  sheaves supported
  at $\infty$ can be described in terms of two positive definite quadratic forms $q\leqslant q'$ on one $\RR$-vector
  space $V$, much in the same way as representing a finite abelian $p$-group
  as quotient of two free $\ZZ_p$-modules of the same rank.
   The role of elementary divisors is then played by the logarithms $\log\lambda_i(q:q')\in\RR_+$,
 of  the eigenvalues of $q$ with respect to $q'$. 

\end{rem}

\vfill\eject

\section{The Hall algebra.}\label{sec:hall-alg}
Let 
\[
Y_n\,\,=\,\, GL_n(\RR)/O_n
\]
be the space of quadratic forms on $\RR^n$. It is a $C^\infty$-manifold of dimension 
$n(n+1)/2$. It is well known that for large $N$ the congruence subgroup
\[
GL_n(\ZZ, N)\,\,=\,\,\bigl\{ \gamma\in GL_n(\ZZ): \,\,\gamma\equiv 1 \mod N\bigr\}
\]
acts on $Y_n$ freely, so $GL_n(\ZZ,N)\backslash Y_n$ is a $C^\infty$-manifold. The set
$\Bunn_n$ is the quotient of this manifold by the finite group $GL_n(\ZZ/N)$
and therefore has a structure of a $C^\infty$-orbifold. In particular, we can speak
about $C^\infty$-functions on $\Bunn_n$.  They are $C^\infty$-functions
on $GL_n(\RR)$,  left invariant under $GL_n(\ZZ)$ and right invariant under $O_n$,
i.e., $C^\infty$-automorphic forms in the classical sense. Let 
$$\label{eq:H-n}
H_n\,=\,C^\infty_0(\Bunn_n)\,\,=\,\,C^\infty_0\bigl( GL_n(\ZZ)\backslash GL_n(\RR)/O_n\bigr)
$$
 be the space of $C^\infty$-functions on $\Bunn_n$ with compact support.
 Consider the direct sum
$$
H\,=\,\bigoplus_{n=0}^\infty H_n, \quad H_0=\CC.
$$
Let $f\in H_m, g\in H_n$. We define their {\em Hall product}
$f*g$ to be the function $\Bunn_{m+n}\to\CC$ given by the formula
\be\label{eq:hall-product}
(f*g)(E)\,\,=\,\,\sum_{E'\subset E} \deg(E')^{n/2} \deg(E/E')^{-m/2} \cdot f(E') g(E/E'),
\ee
where the sum is over all subbundles $E'\subset E$ of rank $m$. 

\begin{prop}\label{prop:hall}
 (a) For every $E$ the sum in \eqref{eq:hall-product} is actually
finite, so $f*g$ is a well defined function.

(b) $f*g$ is again a $C^\infty$-function with compact support.

(c) The operation $f*g$ makes $H$ into a graded associative algebra, with
unit $1\in H_0$.

\end{prop}

We will call the algebra $H$ the {\em Hall algebra of $\SZ$}. In this paper
will be particularly interested in the subalgebra $SH\subset H$ generated by
$H_1=C^\infty_0(\RR_+)$. We will call $SH$ the {\em spherical Hall algebra},
adopting  the terminology of 
\cite{schiffmann-vasserot:higher-genus}, where a similar algebra was studied for
the case of a curve over a finite field. 

\begin{rem}\label{rem:hall-product}

(a) The quantity 
$$\langle E/E', E'\rangle \,\,=\,\,\deg(E')^{n/2} \deg(E/E')^{-m/2}
\,\,= \,\, \sqrt{\deg\,\underline\Hom(E/E', E')   } $$
is the analog of the Euler form used by Ringel 
\cite{ringel} to twist the multiplication in
the Hall algebra of representations of a quiver. In our case, as well as in
the case of curves over a finite field 
\cite{kapranov, schiffmann-vasserot:higher-genus}, twisting by this form simplifies
the form of commutation relations. 

\vskip .2cm

(b) One can get  larger algebras by relaxing
the condition of compact support to
that of sufficiently rapid decay at infinity.  
More generally, there are interesting cases when $f$ and $g$ do not
 have rapid decay, but $f*g$ still makes sense as a convergent series.
 
\end{rem}

 \vskip .2cm

 \noindent {\sl Proof of Proposition \ref{prop:hall}:}
 (a) Since $f$ is with compact support, there is $A>1$
such that $f(E')=0$ unless $\deg(E')\in [1/A, A]$. By Proposition
\ref{prop:finiteness-subbundles} all but finitely many $E'\subset E$
have $\deg(E')<A$, so that the sum in (\ref{eq:hall-product}) is indeed finite.

\vskip .2cm

(b) To see that $f* g$ is smooth, suppose that $E_1$ and $E_2$ are
close to each other in $\Bunn_{m+n}$. Then the corresponding
lattices $L_1$ and $L_2$ are identified in a canonical fashion. Therefore the sets of subbundles 
$E'_1\subset E_1$ and $E'_2\subset E_2$ of rank $m$, are identified,
and so we have a bijection between the sets of summands
in $(f*g)(E_1)$ and $(f*g)(E_2)$. Next, the number of nonzero summands
in both sums is bounded by the same number by the continuity of $f$
and $g$, so we can view $f*g$ as a sum of finitely
many $C^\infty$-functions.

To see that $f*g$ has compact support, let
 $\Sigma_1\subset\Bunn_m$ be a compact set supporting $f$, and
 $\Sigma_2$ be a compact set supporting $g$. For any $E_1\in\Sigma_1$,
 $E_2\in\Sigma_2$ the set of $E\in\Bunn_{m+n}$ that can fit into a
 sequence \eqref{eq:SES},
  is a compact topological space. Indeed, it is the image of a
 continuous map $\Ext^1(E_2, E_1)\to \Bunn_{m+n}$, whose source is
 a compact torus.  Let
 $F$ be the total space of the fibration over $\Sigma_1\times\Sigma_2$
 with fiber over $(E_1, E_2)$ being $\Ext^1(E_2, E_1)$. Then $F$ 
 is compact, while the
 support of $f*g$ is contained in the image of $F$ under a natural
 continuous map into $\Bunn_{m+n}$.
 
 \vskip .2cm
 
 (c) To prove associativity, let 
 $f\in H_{n_1}$, $g\in H_{n_2}$, $h\in H_{n_3}$. Then for
$E\in\Bunn_{n_1+n_2+n_3}$ we have
$$
((f*g)*h)(E)\,\,=\,\,\sum_{E_1\subset E_2\subset E}
d_1^{n_2+n_3\over 2} d_2^{-n_1+n_3\over 2} d_3^{-n_1-n_2\over 2} \cdot 
 f(E_1)\cdot g(E_2/E_1) 
\cdot h(E/E_2),
$$
where $E_1$ runs over subbundles of $E$ of rank $n_1+n_2$, and $E_1$
runs over subbundles of $E_2$ of rank $n_1$, and we have denoted
$$d_1=\deg(E_1),\,\,d_2=\deg(E_2/E_1),\,\,d_3=\deg(E/E_2).$$
 On the other hand
 $$
(f*(g*h))(E)\,\,=\,\,\sum_{E_1\subset E\atop E'_2\subset E/E_1}
\delta_1^{n_2+n_3\over 2} \delta_2^{-n_1+n_3\over 2} \delta_3^{-n_1-n_2\over 2} \cdot 
 f(E_1)\cdot  g(E'_2)
\cdot h\bigl(
(E/E_1)/E'_2\bigr),
$$
where we have denoted 
$$\delta_1=\deg(E_1),\,\, \delta_2=\deg(E'_2),\,\,
\delta_3=\deg((E/E_1)/E'_2).$$
Let $F$ be the set over which the first sum is extended, i.e., the set of
admissible filtrations $E_1\subset E_2\subset E$ with $\rk(E_1)=n_1$
and $\rk(E_2)=n_1+n_2$. Similarly, let $F_2$ be the set over which the second
sum is extended, i.e., the set of pairs $(E_1, E'_2)$, where $E_1\subset E$
is a subbundle of rank $n_1$, and
$E'_2\subset E/E_1$ is a subbundle of rank $n_2$.
We have a map $\rho: F\to F'$ sending $(E_1\subset E_2\subset E)$
into $(E_1, E_2':=E_2/E_1)$. The summand corresponding to any $\phi\in F$
is equal to the summand corresponding to $\rho(\phi)\in F'$. So our statement
reduces to the following.

\begin{lemma}\label{lem:bijection} The map $\rho$ is a bijection. 
\end{lemma}

\noindent {\sl Proof:} An element of $F$ has the form
$$(L_1,V_1,Q_1)\subset (L_2,V_2,q_2)\subset (L,V,q)=E,$$
where $L_1\subset L_2\subset L$ is a filtration by primitive sublattices, and
$q_i=q|_{V_i}$, $i=1,2$. An image of such an element by $\rho$ is the pair
$\bigl( (L_1, V_1, q_1),\ (L'_2, V'_2, q'_2)\bigr)$, where $(L_1, V_1, q_1)$
is as above, while $L'_2=L_2/L_1\subset L/L_1$, and $q'_2=\pi'_*(q_2)$,
with $\pi': V_2\to V_2/V_1=V'_2$ being the canonical projection.

\vskip .2cm

On the other hand, a  general element of $F'$ is a pair
$\bigl( (L_1, V_1, q_1),\ (L'_2, V'_2, q'_2)\bigr)$, where $(L_1, V_1, q_1)$
is as above, while $L'_2\subset L/L_1$ is an arbitrary primitive sublattice
of rank $n_2$, and $V'_2=L'_2\otimes\RR$ and $q'_2$ is the restriction
 to $V'_2$
of the quotient quadratic form $\pi_*(q)$ for the projection $\pi: V\to V/V_1$,
i.e., $q'_2=(i')^*(\pi_*q)$. We
have therefore a Cartesian square of $\RR$-vector spaces
$$
\xymatrix{V_2 \ar[r]^{i} \ar[d]_{\pi'} &V\ar[d]^\pi\cr
V_2'=V_2/V_1 \ar[r]^{i'}&V/V_1
}
$$
with $\pi, \pi'$ surjective and $i, i'$ injective. 
 We claim that
$\pi'_*i^*(q) = i'^*\pi_*(q)$, and hence $\rho(F)=F'$.
  This is a particular case of the following base change property for quadratic forms.

\begin{prop}\label{prop:base-change}
 Let
$$
\xymatrix{U_2 \ar[r]^{i_1} \ar[d]_{j'} &U\ar[d]^j\cr
U'_2 \ar[r]^{i_2}&U
}
$$
be a Cartesian square of $\RR$-vector spaces, such that $i_1, i_2$
are injective and $j,j'$ are surjective. Then for any positive
definite quadratic form $q$ on $U$ we have the equality
$j'_*i_1^* q \,=\, i_2^* j_* q$
of quadratic forms on $U'_2$. 
\end{prop}

\begin{proof} Let $u'_2\in U'_2$. Then
$$(j'_*i_1^*q)(u'_2)\,\,=\,\,\min_{u_2: \, j'(u_2)=u'_2} (i_1^*q)(u_2)\,\,=\,\,
\min_{u_2:\, j'(u_2)=u'_2} q(i_1(u_2)).
$$
Since the square is Cartesian, $i_1$ identifies $(j')^{-1}(u'_2)$
with $j^{-1}(i_2(u'_2))$, so the last minimum is equal to
$$\min_{u:\, j(u)=i_2(u'_2)} q(u)\,\,=\,\,(j_*q)(i_2(u'_2)).$$
\end{proof}
 
  This finishes the proof of 
Lemma \ref{lem:bijection} as well as Proposition \ref{prop:hall}.

\begin{rem} One can extend the definition of  the Hall algebra to the category $\Coh_{\neq\infty}(\SZ)$
defined as in Remark \ref{rem:coherent-sheaves}, using the concept of admissible
exact sequences outlined there. The algebra thus obtained will be a semidirect
product of $H$ and the Hall algebra of the category of finite abelian groups, similarly to
\cite{KSV:curves}, \S 2.6. 

\end{rem}

 \vfill\eject

\section{The Mellin transform.}

A standard tool in the theory of quantum affine algebras is the use of generating
functions, i.e., passing from a collection of coefficients $(c_\alpha)_{\alpha\in\ZZ^n}$
to the Laurent series
\be\label{eq:laurent}
F(t) \,\,=\,\,\sum_{\alpha\in\ZZ^n} c_\alpha t^\alpha, \quad t=(t_1, ..., t_n)\in(\CC^*)^n, \quad
t^\alpha = \prod t_\nu^{\alpha_\nu}.
\ee
This is just the Fourier transform on the free abelian group $\ZZ^n$, but understood
in a more pragmatic way: we do not necessarily restrict to unitary characters
(they form the real torus $|t_i|=1$) but pay attention to the domains
of convergence in the space $(\CC^*)^n$ of all characters. 

\vskip .2cm

 A typical free abelian group to which the above is applied is,
in the Hall algebra approach,
$\Pic(X)/\{\text{torsion}\}=\ZZ$,
where $X$ is a smooth projective curve over $\FF_q$, see
\cite{kapranov, schiffmann-vasserot:higher-genus}, 
 In the present paper the corresponding role is played by the group
 $\Pic(\SZ)=\RR_+$. The Fourier transform on $\RR_+^n$ is known
 as the {\em Mellin transform}. We now give a summary of its properties 
 from the same pragmatic standpoint as above.
 
 \vskip .2cm
 
 Unitary characters of $\RR_{+}^n$ have the form
$$a=(a_1, ..., a_n)\,\longmapsto a^s \,=\, \prod a_\nu^{s_\nu}, \quad 
 \quad s_\nu\in i\,\RR\subset\CC,\,\, a_\nu^{s_\nu}=e^{s_\nu \log(a_\nu)},$$
and the Haar measure is  $d^*a=\prod da_\nu/a_\nu$.
Accordingly, the Mellin transform of a function (or a distribution) $f(a)$ on $\RR_{+}^n$ is
 the  integral
\be\label{eq:mellin}F(s)=(\Mc f)(s) \,\,=\,\,\int_{a\in\RR_+^n}  f(a) a^s d^*a.
\ee
Here, a priori, $s\in i\,\RR^n$, but we are interested in
  allowing the $s_i$  to vary in the complex domain,
i.e., in considering not necessarily unitary characters.  
The group isomorphism
 \be\label{eq:exponential}
\exp: \RR^n\buildrel \sim\over\lra\RR_+^n.
\ee
transforms the Mellin integral into the standard Fourier
integral on $\RR^n$.

\begin{ex}[(Paley-Wiener theorem)]
  If $f(a)$ has compact support, then $(\Mc f)(s)$ converges for any $s\in\CC^n$, i.e., $\Mc f$
is an entire function, analogously to the case of a Laurent series in \eqref{eq:laurent}
being a Laurent polynomial.  
Recall  that an entire function $F(s)$, $s\in\CC^n$, is called a 
{\em Paley-Wiener function}, if there is a constant $B>0$ and, for every
$N>0$ there is a $c_N>0$ such that
$$
|F(s)|\, \leqslant c_N (1+\|s\|)^{-N} e^{B\cdot \|\RRe(s)\|}.
$$
This means, in particular, that $F$ has a faster than polynomial decay on each vertical
subspace $\{\sigma_0+i\RR^n, \, \sigma_0\in\RR^n\}$, while allowed to have exponential growth on
any horizontal subspace. 
We denote by $\PW(\CC^n)$ the space of Paley-Wiener functions on $\CC^n$.
The  Paley-Wiener theorem   says:

\begin{prop} The Mellin transform $\Mc$ identifies $C^\infty_0(\RR_{+}^n)$ with
$\PW(\CC^n)$.
\end{prop}

\begin{proof}
The classical formulation, see, e.g.,  \cite{reed-simon}, Vol. II, Thm. IX.11,
is  for the Fourier transform of compactly supported functions on $\RR^n$.
The case of the Mellin transform reduces to this via  $\exp$. \end{proof}
\end{ex}

 An important point about series \eqref{eq:laurent}
 is that one (meromorphic) function
 can have different Laurent expansions in different regions, while
 the region of convergence of each expansion is ``logarithmically convex",
 i.e., is the preimage of a convex open set $\Delta\subset\RR^n$ under the map
 \[
 \lambda: (\CC^*)^n\lra\RR^n, \quad (t_i) \mapsto (\log |t_i|). 
 \]
 We now review the corresponding features of Mellin expansions. Unlike
 in the case of Laurent series, these features are less familiar, and
 a precise treatment involves L. Schwartz's theory of Fourier transform
 for distributions.

 \vskip .2cm
 
 For a $C^\infty$-manifold or orbifold $M$ we denote by
 $\Dist(M)=C^\infty_0(M)'$ the space of distributions on $M$. 
 Let $\Sc(\RR^n)$ be the space of Schwartz functions on $\RR^n$, and
$\Dc(\RR^n)=\Sc(\RR^n)'\subset\Dist(\RR^n)$ be the dual space of tempered distributions,
see \cite{reed-simon}, Vol. I, \S V.3. 
Recall that a $C^\infty$-function lies in $\Dc(\RR^n)$ if and only if it
has at most polynomial growth.  

We define $\Sc(\RR_+^n)$ and $\Dc(\RR_+^n)$,  the spaces
 of Schwartz functions and  tempered distributions on $\RR_+^n$,
by means of the group isomorphism $\exp$ of 
\eqref{eq:exponential}. 
 For $f\in\Dc (\RR_+^n)$ we define $\Mc f$
to be the tempered distribution on $i\RR^n$ given by the Fourier-Schwartz
transform of $f\circ\exp$. 

\vskip .2cm

For a distribution $f\in \Dist(\RR_+^n)$ we denote by $\Temp(f)$ and call the
{\em tempering set} of $f$,  the set of $\sigma\in\RR^n$ such that $f(a)\cdot a^\sigma$ is
a tempered distribution. It is known (see \cite{reed-simon}, Vol. II, Lemma after Th. IX. 14.1)
that $\Temp (f)$ is a convex subset in $\RR^n$.  We say that $f$ is {\em temperable}, if
$\Temp(f)$ has non-empty interior. For any convex open set  $\Delta\subset\RR^n$
we denote by $U_\Delta = \{ s\in\CC^n |\, \RRe(s)\in\Delta\}$  the corresponding tube domain.

 \begin{prop}\label{prop:mellin-convergence}
 Let $f$ is a temperable distribution on $\RR_+^n$, and $\Delta$ be the interior of $\Temp(f)$.
Then $F(s) = (\Mc f)(s)$ is an analytic function in $U_\Delta$, which
has an at most  polynomial growth on each vertical subspace $\sigma_0+i\RR^n$, $\sigma_0\in\Delta$.
\end{prop}

\begin{proof}  To see holomorphy, it is enough to assume that $0$ is an interior
point of $T(f)$ and to show that $\Mc f$ is holomorphic in an open neighborhood of $i\RR^n$. 
 For a sequence of signs $\varepsilon=(\varepsilon_1, ..., \varepsilon_n)$,
$\varepsilon_i\in\{\pm 1\}$ let $(\RR_+^n)_\varepsilon\subset\RR_+^n$ be the domain
given by conditions $a_i^{\epsilon_i}>1$, and $\CC^n_\varepsilon\subset \CC^n$ 
be given by the condition
$\varepsilon_i \RRe(s_i)<0$. Let $\Mc_\varepsilon f$ be the partial Mellin integral
of $f$, taken over $(\RR_+^n)_\varepsilon$. 
  If $s\in \CC^n_\varepsilon$, then the function $a^s$ 
decays exponentially at the infinity of
$(\RR^n_+)_\varepsilon$. Therefore, if $f$ is a tempered distribution on $\RR_+^n$
(i.e., if $0\in \Temp(f)$), then
 $\Mc_\varepsilon f$
extends to a holomorphic function in $\CC^n_\varepsilon$. If, moreover, $0$ is an interior point
of $\Temp(f)$, then $\Mc_\varepsilon f$ is holomorphic in additive translates $\CC^n_\varepsilon+\sigma$
for $\sigma$ running in an open neighborhood of 0 in $\RR^n$. Therefore $\Mc f=\sum_\varepsilon
\Mc_\varepsilon f$ is holomorphic 
for $\RRe(s)$ running in some open neighborhood of 0, as claimed. 

To see that $\Mc f$ has at most polynomial growth on each $\sigma_0+i\,\RR^n$, it is again
enough to treat the case $\sigma_0=0$. The restriction of $\Mc f$ to $i\,\RR^n$ is 
a tempered distribution, the Fourier-Schwartz
transform of   $f\circ\exp$. As it is also a real analytic function,
it must be of polynomial growth.
\end{proof}

Next, we discuss the inverse Mellin transform, i.e., the analog of
the formula which finds each coefficient $c_\alpha$ in
\eqref{eq:laurent} as an integral of $F(t)$ times a monomial. 
Formally, the inverse Mellin integral is defined by
 \be\label{eq:inverse-mellin} 
 f(a) =  (\Nc_\Delta F)(a) \,\,=\,\,{1\over(2\pi i)^n}\int_{s\in \sigma_0+ i\RR^n}  F(s) a^{-s} ds. 
\quad \sigma_0\in\Delta,\,\,a\in\RR_+^n,
\ee
 In our case, this integral should again be understood using Schwartz's theory. More
precisely, we have:

\begin{prop}\label{prop:mellin-inversion}
  Let $\Delta\subset\RR^n$ be a convex open set and 
 $F(s)$ be an analytic function in $U_\Delta$ with at most polynomial growth
 on each vertical subspace. Choose $\sigma_0\in\Delta$ and define 
 $ f(a) =  (\Nc_\Delta F)(a)$  as  $a^{-\sigma_0}$ times the inverse Fourier transform of $g$ as a tempered distribution on 
  $\sigma_0+i\,\RR^n\simeq \RR^n$ (the Fourier transform being transplanted to $\RR_+^n$ via $\exp$).
   Then $\Nc_\Delta F$
   is  independent on $\sigma_0\in\Delta$, and is  a temperable distribution on $\RR_+^n$
 such that $\Delta\subset \Temp(f)$ and $\Mc f=F$. 
\end{prop} 

We will call $\Nc_\Delta(F)$ the {\em coefficient function} of $F$ in $U_\Delta$.
Thus  the existence of the coefficient function presupposes that $F$ grows
at most polynomially  on each vertical subspace in $U_\Delta$. 
As usual with the Fourier transform, the coefficient function of the product of
analytic functions is the convolution (on the group $\RR_+^n$) of the coefficient
functions of the factors.

\begin{proof}
To show independence, it is enough to assume $0\in\Delta$ and compare the integrals
\eqref{eq:inverse-mellin} over $i\,\RR^n$ and $\sigma_0+i\,\RR^n$ for $\sigma_0$ being close to
$0$ in $\Delta$. Both functions $F(s)$ and $F(s+\sigma_0)$ are tempered distributions on
$i\,\RR^n\simeq \RR^n$ and so have Fourier-Schwartz  transforms. Moreover, 
  $F(s+\sigma_0)$  the sum of a Taylor series involving
derivatives of $F(s)$ (evaluated on $i\,\RR^n$).  So the Fourier transform of $F(s +\sigma_0)$
 is product of the Fourier transform of $F(s)$
 and an exponential factor. This factor is accounted
for by the change in $a^s$ in the integral \eqref{eq:inverse-mellin}, showing the independence.
The remaining claims follow from the inversion theorem for the Fourier-Schwartz transform.
\end{proof}

Let us note the particular case $\Delta=\RR^n$.

\begin{cor}\label{cor:abs-tempered}
The Mellin transforms $\Mc$ and $\Nc$ defines mutually inverse isomorphisms
 between the following two spaces:
\begin{itemize}
\item $\Dc(\RR_+^n)_{\on{abs}}$, the space of {\em absolutely tempered distributions},
i.e., of distributions $f(a)$ such that $f(a) a^s$ is tempered for each $s\in\CC^n$. 
\item $\Oc(\CC^n)_{\on{pol}}$, the space of entire functions in $\CC^n$ with
at most polynomial growth on each vertical subspace.
\end{itemize}
\qed
\end{cor}

Note that an absolutely tempered distribution has actually exponential decay
at the infinity of $\RR_+^n$. 

\vskip .2cm

For future reference we recall two elementary properties of  the Mellin/Fourier transform.
We denote by $\delta_c\in\Dc(\RR_+)$ the delta function at $c\in\RR_+$. 

\begin{prop}\label{prop:mellin-delta-function}
(a) Let $F(s)$ be analytic in $U_\Delta$, with the coefficient function $f(a) = \Nc_\Delta(F)$.
Then for any $\nu=1, ..., n$ we have
\[
\Nc_\Delta(s_\nu F(s))\,\,=\,\, -a_\nu {d\over da_\nu} f(a). 
\]

(b) Let $\Delta$ be an interval $(c,c')\subset \RR$, so $U_\Delta$ is a strip in $\CC$. Let 
 $h(s)$ be analytic in $U_\Delta$, with coefficient function
   $k(a)$, $a\in\RR_+$. 
Consider the function of two variables 
\[
F(s_1, s_2) = h(s_1-s_2), \quad (s_1, s_2)\in U_{\widetilde \Delta}= \{ c< \RRe(s_1-s_2) < c'\}
\]
  Then the coefficient function  of $F$ is found by
  \[
  (\Nc_{\widetilde\Delta}F)(a_1, a_2) \,\,=\,\,\delta_1(a_1a_2)\cdot k(a_1).
  \]
  \qed
 
\end{prop}

\begin{ex} 
 Let $\zeta(s)$ be the Riemann zeta function, and
 \be\label{zeta}
 \zeta^*(s) \,\,=\,\,\pi^{-s/2}\Gamma(s/2)\zeta(s)
 \ee
 be the zeta function of $\SZ$. It is a meromorphic function on $\CC$ with simple
 poles at 0 and 1, satisfying $\zeta^*(s)=\zeta^*(1-s)$. 

\vskip .2cm

 The function $\Gamma(s)$ has exponential decay on each vertical line $\sigma_0+i\RR$,
as follows from the Stirling formula. The function $\zeta(s)$ has at most polynomial growth on
each vertical line, see \cite{edwards}, Ch.9. Therefore $\zeta^*(s)$ has
exponential decay on each vertical line and therefore has a well defined coefficient
function in each of the three strips of holomorphy: $\RRe(s)>1$, $0<\RRe(s)<1$ and $\RRe(s)<0$.
The coefficient function in $\RRe(s)>1$ is given by the classical formula of Riemann,
see \cite{edwards}, \S 1.7:
\be\label{eq:zeta-theta}
(\Nc_{\RRe(s)>1}\zeta^*)(a) \,\,=\,\,\theta(a^2)-1, \quad \theta(b) \,\, := \,\,
\sum_{n=-\infty}^\infty e^{-n^2\pi  b}.
\ee
It can be obtained by forming the convolution (on  $\RR_+$) 
of the distribution $\sum_{n=1}^\infty \delta_{1/n}$, of 
the function
$2e^{-a^2}$ and of the distribution $\delta_{1/\sqrt{\pi}}$. These three distributions 
 are the coefficient
functions for  $\zeta(s)=\sum 1/n^s$, for $\Gamma(s/2) = 2\int_0^\infty e^{-a^2} a^s d^*a$
and for $\pi^{-s/2}$ respectively.
 The coefficient functions in the two other strips are obtained by moving the contour past the
poles of $\zeta^*(s)$ at $s=1$ and $s=0$ with residues $\pm 1/\sqrt{\pi}$:
\[
\begin{gathered}
(\Nc_{0<\RRe(s)<1} \zeta^*)(a) = \theta(a^2)-1-{1\over a\sqrt{\pi}}, \cr
(\Nc_{\RRe(s)<0} \zeta^*)(a)=\theta(a^2)+1/\sqrt{\pi}-1-{1\over a\sqrt{\pi}}. 
\end{gathered}
\]

\end{ex}

\vfill\eject
 
\section{The zeta function shuffle algebra.}

We recall the formalism of shuffle algebras of Feigin-Odesskii \cite{feigin-odesskii:shuffles},
see \cite{schiffmann-vasserot:higher-genus}
 \cite{KSV:curves} for a more systematic discussion
in the rational function case. We denote by $\Sen_n$ the symmetric group
of permutations of $\{1, ..., n\}$. 

\vskip .2cm

Let $\varphi(s)$ be a meromorphic function on $\CC$. 
For any $m,n>0$ let 
$Sh(m,n)$ be  the set of $(m,n)$-{\em shuffles}, i.e., permutations $w\in \Sen_{m+n}$ such that $w(i)<w(j)$
whenever $i<j$ and either both $i,j\in [1,m]$ or both $i,j\in[m+1,m+n]$.
For any $w\in Sh(m,n)$ consider the following  meromorphic function on $\CC^{m+n}$:
\be\label{eq:phi-sigma}
\varphi_w(s_1, ..., s_{m+n}) \,\,=\,\, \prod_{{i\in[1,m]\atop j\in [m+1,m+n],}
\atop w(i)>w(j)} \varphi(s_i-s_j).
\ee

Let  $\Oc(\CC^n)\subset \Mer(\CC^n)$  be the spaces  of all entire and meromorphic
 functions on $\CC^n$ (defined to be
equal to $\CC$ for $n=0$). 
On  the direct sum $\bigoplus_n \Mer(\CC^n)$
we introduce
  the {\em shuffle multiplication}   
    \be\label{eq:shuffle-product-1}
 \sm_{m,n}: \Mer(\CC^m)\otimes \Mer(\CC^n) \lra \Mer(\CC^{m+n}), 
 \quad 
 F\otimes F'\mapsto F\sm F',  
  \ee
 by the formula
 \be\label{eq:shuffle-product-2}
 \begin{gathered}
  (F\sm F')(s_1, ..., s_{m+n}) \,\,=
\cr = \sum_{w\in Sh(m,n)}w\biggl(
  F(s_{(1)}, ..., s_{(m)})
F'(s_{m+1}, ..., s_{m+n})\biggr)   \cdot\varphi_w(s_1, ..., s_{m+n}).
 \end{gathered}
 \ee
 The following is then straightforward, as in 
 \cite{feigin-odesskii:shuffles}.
 
  \begin{prop} 
  The shuffle multiplication $\sm$ makes $\bigoplus_n \Mer(\CC^n)$ into
 a graded  associative algebra, with unit $1\in\Mer(\CC^0)$. \qed
  \end{prop}
  
  \vskip .2cm
 
 Assume further that the function $\varphi$ satisfies the equation 
    $\varphi(-s)\varphi(s)=1$, and, moreover, is represented in the form
    \be\label{eq:antisymmetrization}
 \varphi(s) = \lambda(s)^{-1}\lambda(-s)   
      \ee
for some meromorphic function $\lambda(s)$. For $n\geqslant 0$ let $\Mer(\CC^n)^{\Sen_n}$
be the space of symmetric meromorphic functions on $\CC^n$.  On the
direct sum $\bigoplus_n \Mer(\CC^n)^{\Sen_n}$, 
we introduce the {\em symmetric shuffle multiplication}
\be\label{eq:symmetric-shuffle1}  
\begin{gathered}
 \star_{m,n}: \Mer(\CC^m)^{\Sen_m}\otimes \Mer(\CC^n)^{\Sen_n} \lra \Mer(\CC^{m+n})^{\Sen_{m+n}},  \cr
F\otimes F'\mapsto F\star F', 
\end{gathered}
\ee
 by the formula
\be\label{eq:symmetric-shuffle2}
\begin{gathered}
(F\star F') (s_1, ..., s_{m+n}) \,\,=  \cr
\,\,\sum_{w\in Sh(m,n)} w\biggl(
F(s_1, ..., s_m) F'(s_{m+1}, ..., s_{m+n})
\prod_{1\leqslant i\leqslant m\atop m+1\leqslant j\leqslant m+n} \lambda(s_i-s_j)\biggr).
\end{gathered}
 \ee
 
 \begin{prop}\label{prop:symmetric-shuffle}
(a) The multiplication $\star$ makes $\bigoplus_n \Mer(\CC^n)^{\Sen_n}$ into a
graded associative algebra
with unit.

(b) The correspondence
$$F(s_1, ..., s_n) \,\,\longmapsto F(s_1, ..., s_n)\prod_{i<j} \lambda(s_i-s_j)$$
defines an injective algebra homomorphism 
$$\biggl(\bigoplus_n \Mer(\CC^n)^{\Sen_n}, \star\biggr)\,\,
\hookrightarrow\,\, \biggl(\bigoplus_n \Mer(\CC^n), \sm \biggr).$$

(c) Assume that $\lambda(s)$ has no poles except, possibly, a first order pole at $s=0$.
Then the graded subspace $\bigoplus_n \Oc(\CC^n)^{\Sen_n}$ is a subalgebra with
respect to $\star$. 
 
\end{prop}

\begin{proof} Parts (a) and (b) are proved straightforwardly, as
in  \cite{feigin-odesskii:shuffles}. For (c), let us indicate why
\[
\star_{1,1}: \Oc(\CC)\times\Oc(\CC)\lra\Mc er(\CC^2)
\]
takes values in $\Oc(\CC^2)$ (the general case is similar). 
Writing $\lambda(s) = cs^{-1} + h(s)$ with $h$ entire, we have, for $f,g\in\Oc(\CC)$:
\[
\begin{split}
(f\star g)(s_1, s_2) \,\,= \,\,
\lambda(s_1-s_2) f(s_1)g(s_2) + \lambda(s_2-s_1) f(s_2) g(s_1) \cr
=\,\, \frac{c}{s_1-s_2}\bigl[ f(s_1)g(s_2)-f(s_2)g(s_1)\bigr] \,\,+\,\,\text{(entire)},
\end{split}
\]
and the expression in square brackets, being an entire antisymmetric function,
vanishes on the diagonal $s_1=s_2$. 

\end{proof}

\begin{df} (a) We call the  {\em shuffle algebra} associated to $\varphi$  the subalgebra
 $\Shuff(\varphi)\subset \bigoplus_{n\geqslant 0} \Mer(\CC^n)$   generated by
the space $\Oc(\CC)\subset\Mer(\CC^1)$.  
We call the {\em  symmetric shuffle algebra} associated to $\lambda$ the subalgebra
$\SShuff(\lambda) \subset \bigoplus_{n\geqslant 0} \Mer(\CC^n)^{\Sen_n}$ 
generated by $\Oc(\CC)$.

(b) The {\em Paley-Wiener shuffle algebra} $\Shuff(\varphi)_{\PW}$, resp. the 
{\em Paley-Wiener
symmetric shuffle algebra} $\SShuff(\lambda)_{\PW}$, is defined as the subalgebra
in $\Shuff(\varphi)$, resp. $\SShuff(\lambda)$, generated by the subspace
$\PW(\CC)\subset\Oc(\CC)$. 
\end{df}

Thus, if $\varphi$ and $\lambda$ are related by \eqref{eq:antisymmetrization},
then $\Shuff(\varphi)$ is isomorphic to $\SShuff(\lambda)$
and $\Shuff(\varphi)_{\PW}$  to $\SShuff(\lambda)_{\PW}$
 If, further, $\lambda$
satisfies the condition (c) of Proposition \ref{prop:symmetric-shuffle},
then $\SShuff(\lambda)$ is a subalgebra of $\bigoplus \Oc(\CC^n)^{\Sen_n}$. 

\vskip .3cm
 
 We now specialize $\varphi(s)$ to be   the following meromorphic function:
 \be\label{eq:Phi}
 \Phi(s) \,\,=\,\,\zeta^*(s)/\zeta^*(s+1).
 \ee
 It is known as the {\em global Harish-Chandra function}
 (or the {\em scattering matrix}) for $\SZ$, cf. \cite[\S 7]{lax-phillips}.
  The functional equation for $\zeta(s)$ implies that
$\Phi(-s)\Phi(s) =1.$ We also consider the function
\be
\Lambda(s) = \zeta^*(-s)(s-1)(-s-1).
\ee
It has just one simple pole at $s=0$, with $\res_{s=0}\Lambda(s)=1$,
 and zeroes at nontrivial zeroes of $\zeta(s)$
as well as at $s=-1$. We also have the identity
\be
\Phi(s) =  \Lambda(s)^{-1}\Lambda(-s).
\ee

Here is the  main result of this paper, which will be proved in Section \ref{sec:CT}.

\begin{thm}\label{thm:main}
 The Mellin transform $\Mc: SH_1=C^\infty_0(\RR_+)\buildrel\sim\over
\to \PW(\CC)$ extends to an isomorphism of algebras 
$SH\to \Shuff(\Phi)_{\PW}\simeq \SShuff(\Lambda)_{\PW}$.
 \end{thm}

 The bigger algebra $\SShuff(\Lambda)\simeq\Shuff(\Phi)$ can be thus seen as a natural
completion of $SH$.

\vfill\eject

\section{The constant term and its Mellin transform.}\label{sec:CT}

The sum over shuffles appearing in the definition of the shuffle algebra
turns out to match quite exactly the sum over shuffles appearing
in the classical formula for the constant term of a (pseudo-)Eisenstein series,
cf.  \cite{moeglin-waldspurger-book}, II.1.7.
In this section we perform a detailed comparison and obtain a proof
of  Theorem \ref{thm:main}.  Our comparison
 can be  organized into 5 steps: 

\begin{enumerate}[(A)]

\item  Taking the  constant term of an automorphic form on $GL_n$ with respect to the Borel subgroup 
$B_n$,
defines a map
$\CT_n: H_n\to C^\infty(\RR_+^n)$.  

\item 
We denote by 
 $\widetilde\CT_n$ the twist of 
  $\CT_n$ by the analog of the Euler
form (Iwasawa Jacobian) to match the formula \eqref{eq:hall-product} for the Hall product. 
 It is then adjoint to the Hall multiplication map
$$*_{1^n} = *_{1,...,1}: H_1^{\otimes n}\lra H_n$$
with respect to  natural positive definite Hermitian scalar products on both sides. 
This adjointness implies that the restriction of $\widetilde\CT_n$ to $SH_n=\Im(*_{1^n})$
is an embedding $SH_n\to C^\infty(\RR_+^n)$.

\item The standard principal series intertwiners for $GL_n$ give rise to integral operators
$$M_w: C^\infty_0(\RR_+^n) \lra C^\infty(\RR_+^n), \quad w\in\Sen_n,$$
whose domain of definition can be extended to
include more general functions. The  formula for the constant term of a
pseudo-Eisenstein
series then says:
\be\label{eq:CT-Eis}
\widetilde\CT_{n'+n''}(f' *f'') =\sum_{w\in Sh(n', n'')} M_w(\widetilde\CT_{n'}(f')\otimes
\widetilde\CT_{n''}(f'')), \quad f'\in H_{n'}, f''\in H_{n''}.
\ee
 
\item For $f\in SH_n$ we define $Ch_n(f)$ to be the Mellin transform of $\widetilde \CT_n(f)$. 
It is  verified to represent a meromorphic function on $\CC^n$. Taken together, the maps
$Ch_n$ define then an embedding of vector spaces $Ch: SH\to\bigoplus_n \Mer(\CC^n)$. 

\item Finally, one sees that the Mellin transform takes $M_w$ to the operator on
$\Mer(\CC^n)$ taking  a function $F(s_1, ..., s_n)$ to
$$(w F)(s_1, ..., s_n) \cdot \prod_{i<j\atop w(i)>w(j)} \Phi(s_i-s_j),$$
and so $Ch$ takes the Hall product into the shuffle product, by comparing 
\eqref{eq:CT-Eis} with \eqref{eq:shuffle-product-2}.

\end{enumerate}

\noindent  We now  implement each step in detail. 

\paragraph{A. The constant term.} 
We will use both the real and the adelic interpretation of the component $H_n$ of $H$: 
$$
H_n \,\,=\,\,C^\infty_0 \bigl(GL_n(\ZZ)\backslash GL_n(\RR)/O_n\bigr)\,\,
=\,\,C^\infty_0\bigl(GL_n(\QQ)\backslash GL_n(\Aen)/GL_n(\widehat\Oc)\bigr).
$$
 Let $B=B_n$  be the  lower triangular
Borel subgroup in $GL_n$ and 
  $U$ be the unipotent radical of $B$. For 
$f\in H_n$ its {\em constant term} is the function $\CT(f)$ on $\RR_+^n$ defined
in either interpretation by:
\be\label{eq:CT-untwisted}
\begin{gathered}
\CT_n(f)(a_1, ..., a_n) \,\,=\,\,\int_{u\in U(\ZZ)\backslash U(\RR)} f\bigl (u\cdot\diag(a_1, ..., a_n)\bigr) du\,\, =  \cr
=\,\, \int_{u_\Aen\in U(\QQ)\backslash U(\Aen)} f\bigl (u_\Aen\cdot\diag(a_1, ..., a_n)\bigr) du_\Aen,
\quad a_i\in\RR_+.
\end{gathered}
\ee
Here $du$, resp. $du_\Aen$, is the Hall measure on $U(\RR)$, resp. $U(\Aen)$,
 normalized so that $U(\ZZ)\backslash U(\RR)$, resp. $U(\QQ)\backslash U(\Aen)$, has volume 1.
Clearly, $\CT_n(f)$ is a $C^\infty$-function on $\RR_+^n$, bounded by $\max |f(g)|$.  

\begin{prop}\label{prop:CT-support}
 For every $f\in H_n$   there is $c\in\RR_+$
such that $\Supp(\CT_n(f))$ is contained in the domain
$$a_1\leqslant c, \,\, a_1a_2\leqslant c, \,\,\cdots, \,\, a_1... a_{n-1}\leqslant c,
\,\,\, {1\over c}\leqslant a_1\cdots a_n\leqslant c.$$
\end{prop} 

\begin{proof}  For $(a_1, ..., a_n)\in \RR_+^n$ and $u\in U(\RR)$ let
$V(a_1, ..., a_n; u)$ be the vector bundle on $\SZ$ associated to the class of
$u\cdot\diag(a_1, ..., a_n)$ in the double quotient. This bundle has a canonical 
admissible filtration
$$V_1\subset V_2\subset\cdots \subset V_n \,\,=\,\, V(a_1, ..., a_n;u)$$
with $\rk(V_i)=i$ and $V_i/V_{i-1}\simeq \Oc(a_i)$. 
 But given any vector bundle $V$ on $\SZ$, there is $c\in \RR_+$ such that
for any admissible filtration $V_1\subset \cdots \subset V_n=V$ with $\rk(V_i)=i$,
the numbers $a_i=\deg(V_i/V_{i-1})$ satisfy the conditions of Proposition \ref{prop:CT-support}.
This follows from Proposition
 \ref{prop:finiteness-subbundles}, and we can clearly find a common $c$ for bundles varying in a compact
 subset of $\Bunn_n$.   
\end{proof}

\paragraph{B. Twisted constant term and its adjointness.} 
Let
\[
dg\,\,=\,\,\frac{\prod_{i,j=1}^n dg_{ij} } {\det(g)^n}, \quad d^*a \,=\,\prod_{i=1}^n
\frac{da_i}{a_i}
\]
be the standard Haar measures on $GL_n(\RR)$ and $\RR_+^n$. 
We    introduce notation for the factors in the Iwasawa
 decomposition:
  $$GL_n(\RR)=U\cdot \RR_+^n\cdot O_n, \quad g= u\cdot a\cdot k, \quad 
  a=(a_1, ..., a_n).$$
  We write $a=a(g)$, $a_\nu=a_\nu(g)$ etc. as functions of $g\in GL_n(\RR)$. 
  Let $dk$ be the Haar measure on $O_n$ of volume 1. 
  
  The Haar measure $dg$ on $GL_n(\RR)$
  has, in Iwasawa coordinates, the  well known form
   \be\label{eq:haar-iwasawa}
   dg=\delta(a) du\cdot dk \cdot d^*a,  
      \ee
   where the {\em Iwasawa Jacobian} $\delta(a)$ is defined by
   \be\label{eq:iwasawa-jac}
   \delta(a) \,\,= \delta_n(a)\,\,=\,\,\prod_{1\leqslant i<j\leqslant n} \frac{a_j}{a_j}
   \,\,=\,\, \prod_{i=1}^n a_i^{-n+2i-1}. 
   \ee
    See, e.g., \cite{terras}, \S 4.1, Exercise 20 for upper-triangular matrices.
    We also write $\delta_n(g)=\delta_n(a(g))$ for $g\in GL_n(\RR)$. 

\vskip .2cm

Let us  make $H_n$ and $C^\infty_0(\RR_+^n)\supset H_1^{\otimes n}$  into 
 pre-Hilbert spaces via the positive 
 definite Hermitian scalar products
$$
(f_1, f_2)_H \,\,=\,\,\int_{GL_n(\ZZ)\backslash GL_n(\RR)} f_1(g)\overline {f_2(g)} dg,\quad
(\varphi_1, \varphi_2) \,\,=\,\,{1\over 2^n} \int_{\RR_+^n} \varphi_1(a) \overline{\varphi_2(a)}d^*a.
$$
More generally, in each case the scalar product makes sense whenever only one
of the arguments has compact support. 
Define the {\em twisted constant term} of $f\in H_n$ to be the function
\be\label{eq:thwisted-CT}
 \widetilde{\CT}_n(f)(a_1, ..., a_n) \,\,=\,\,\CT(f)(a_1, ..., a_n) \cdot \delta(a)^{1/2}.
\ee

\begin{prop} The map $\widetilde\CT_n: H_n\to C^\infty(\RR_+^n)$
 is adjoint to $*_{1^n}: H_1^{\otimes n}\to H_n$, i.e., we have
\[
(*_{1^n}(\varphi), f)_H \,\,=\,\,(\varphi, \widetilde\CT_n(f)), 
\quad \varphi\in  H_1^{\otimes n},\,\, f\in H_n. 
\]
\end{prop}

\begin{proof} This is  standard, 
 we   provide details for   convenience of the reader. 
 For $\varphi\in C^\infty_0(\RR^n_+)$ we define a function
     $\widetilde\varphi$ on $U\backslash GL_n(\RR)/O_n$
by 
$$\widetilde\varphi(g) \,\,=\,\,\varphi(a_1(g), ..., a_n(g))\cdot \delta(g)^{-1/2}.$$
Translating the (iterated) formula \eqref{eq:hall-product} for the Hall product, into group-theoretical terms,
we have
$$(*_{1^n} (\varphi)) (g)\,\,=\,\,\sum_{\gamma\in B_n(\ZZ)\backslash GL_n(\ZZ)} \widetilde\varphi(\gamma g)
$$
(a pseudo-Eisenstein series). 
The adjointness then follows from the expression of $dg$ in terms of the Iwasawa factorization:  
 \[
\begin{split}
(*_{1^n}(\varphi), f)_H \,\,= 
\int_{g\in GL_n(\ZZ)\backslash GL_n(\RR} \overline{f(g)} \sum_{\gamma\in B(\ZZ)\backslash
GL_n(\ZZ)} \widetilde\varphi(\gamma g) dg \cr
=\int_{x\in B(\ZZ)\backslash GL_n(\RR)} \overline{f(x)}\widetilde\varphi(x) dx \,\, 
\buildrel\operatorname{def}\over =\,\,
\int_{x\in B(\ZZ)\backslash GL_n(\RR)} \overline{f(x)}\varphi(x) \delta(x)^{-1/2}dx
\cr
=\,\,{1\over 2^n}
\int_{y\in U(\ZZ)\backslash GL_n(\RR)} \overline{f(y)} \varphi(y)\delta(y)^{-1/2} dy \cr
=\,\,{1\over 2^n} \int_{ z \in U(\RR)\backslash GL_n(\RR)} 
\int_{u\in U(\ZZ)\backslash U(\RR)} \overline{f(u z)}
 \varphi( z)\delta(z)^{-1/2}
du d z
\cr
\,\,\buildrel \eqref{eq:haar-iwasawa}\over = \,\, {
1\over 2^n}\int_{a\in\RR_+^n} \overline{{\CT}_n(f)(a)} \varphi(a)
\delta(a)^{+1/2} d^*a
\,\,=\,\, (\varphi, \widetilde\CT_n(f)).
\end{split}
\]
\end{proof}

\begin{cor} The map $\widetilde\CT_n: SH_n\to C^\infty_0(\RR_+^n)$ is
injective. 
\end{cor}

\begin{proof}
 By definition of $SH$ as the subalgebra
generated by $H_1$, a  non-zero element $f\in SH_n$ has the form 
$f=*_{1^n}(\varphi)$ for some $\varphi\in H_1^{\otimes n}$.   We can regard
$\varphi$ as an element of 
 $C^\infty_0(\RR_+^n)$.  
  To prove that $\widetilde \CT_n(*_{1^n}(\varphi))\neq 0$, we notice that by adjointness and  by
   the positivity of the
scalar product on $H$, we have
$$\bigl (\varphi,\widetilde \CT_n(*_{1^n}(\varphi))\bigr) \,\,=\,\, \bigl( *_{1^n}(\varphi), *_{1^n}(\varphi)\bigr)_H
\,\,=\,\,
(f,f)_H
 \,\, > \,\, 0.$$
\end{proof}

\paragraph{C. The principal series intertwiners.}  We use the intertwiners in their adelic form,
as this form accounts for the appearance of the factors involving the Riemann zeta
in  the function $\Phi(s)$  defining the shuffle algebra, see \eqref{eq:Phi}. 

\vskip .2cm

Let $A_n$ be the diagonal subgroup in $GL_n$. We have the identification
$$\RR_+^n\,\,=\,\, U(\Aen) A_n(\QQ)\bigl\backslash GL_n(\Aen)\bigr/ K_n, \quad K_n = O_n \prod_p GL_n(\ZZ_p).$$
For $w\in\Sen_n$ let $U_w = U\cap (w^{-1}Uw)$. Using the above identification,
we define the operator
$$M_w: C^\infty_0(\RR_+^n) \lra C^\infty(\RR_+^n), \quad (M_w \varphi)(g) =
 \int_{u\in (U(\Aen)\cap U_w (\Aen))  \backslash U(\Aen)} \varphi(w ug) du, 
 $$
 cf.
 \cite{moeglin-waldspurger-book}, II.1.6. More generally, $M_w(\varphi)$ can be defined
if, for any $g$, the function $u\mapsto \varphi(wug)$ on
the domain of integration has sufficiently fast decay (for example, has compact support). 
Here is an example, to be used later.

\vskip .2cm 

 We consider the following domain in $\CC^n$:
 \be\label{eq:C-n>}
 \CC^n_>\,\,=\,\,\bigl\{ s=(s_1, ..., s_n): \, s_\nu-s_{\nu+1}>1, \nu=1, ..., n\bigl\},
 \ee
 where we put $s_{n+1}=0$. For $w\in \Sen_n$ put
 \be
 \Phi_w (s) \,\,=\,\,\prod_{1\leqslant i<j\leqslant n\atop w(i)>w(j)} \Phi(s_i-s_j).
 \ee

 \begin{prop}\label{prop:M-w-a-s}
 If $s=(s_1, ..., s_n)\in\CC^n_>$, then applying $M_w$ to the function $a\mapsto a^s$
 gives a convergent integral, and it is found as follows:
 \[
 M_w(a^s) \,=\, a^{w(s)} \Phi_w(s). 
 \]
  \end{prop}
  
  \begin{proof} This  is a version of the classical Gindikin-Karpelevich formula.  
  More precisely, the 
  value of the adelic intertwiner is found as the
  Euler product of the values of similarly defined local intertwiners
  (involving the integration over the $p$-adic or real group).
  Each  local integral is found by Gindikin-Karpelevich to
  contribute the factor
  \[
 \prod_{1\leqslant i<j\leqslant n\atop w(i)>w(j)} \frac{\zeta_p(s_i-s_j)}{\zeta_p(s_i-s_j+1)},
  \]
    where $\zeta_p$ is the $p$th Euler factor of the Riemann zeta, or the Gamma factor for
    $p=\infty$. 
  \end{proof}
 
 \vskip .2cm
 
 For $\varphi'\in C^\infty(\RR^{n'}_+)$ and $\varphi''\in C^\infty(\RR^{n''}_+)$
we define $\varphi'\otimes\varphi''\in C^\infty(\RR^{n'+n''}_+)$ by
\be\label{eq:otimes-twisted}
(\varphi \otimes\varphi'')(a_1, ..., a_{n'+n''}) \,\,=\,\,\varphi'(a_1, ..., a_{n'})\varphi''(a_{n'+1}, ..., a_{n'+n''}).
\ee
We will use similar notation in other situations without special explanation. 
 
\vskip .3cm

Having now defined all the ingredients of the equality
   \eqref{eq:CT-Eis}, we explain how it is proved. This is again a standard
argument,  using the Bruhat decomposition of  a Grassmannian into cells
labelled by shuffles, cf. \cite{moeglin-waldspurger-book}, II.1.7
for the case of any parabolic subgroup in any reductive group.

\vskip .2cm

To give some details in our particular case, 
let $n=n'+n''$ and 
$P_{n',n''}\subset GL_n$ be the parabolic (block-lower-triangular) subgoup corresponding to $(n',n'')$.
We  denote $U_{n',n''}$ its unipotent radical and $A_{n',n''}=GL_{n'}\times GL_{n''}$ the Levi subgroup. 
Then the Iwasawa decompostion implies that

  \be\label{eq:parabolic-iwasawa}
  \begin{gathered}
  \bigl( GL_{n'}(\QQ)\backslash GL_{n'}(\Aen)/ K_{n'}\bigr) \times 
\bigl( GL_{n'}(\QQ)\backslash GL_{n'}(\Aen)/ K_{n'}\bigr)\,\,\buildrel\sim\over\lra
\cr
\buildrel\sim\over\lra
 \,\, \bigl(U_{n',n''}(\Aen) A_{n',n''}(\QQ)\bigr)\bigl\backslash
GL_n(\Aen)\bigr/K_n.
\end{gathered}
\ee

\vskip .2cm

Given $f'\in H_{n'}$, $f''\in H_{n''}$, let $f$ be the function on the right hand side of
\eqref{eq:parabolic-iwasawa} corresponding to the function
 $$
(g',g'') \,\,\longmapsto\,\, |\det(g')|^{n''/2}\cdot |\det(g'')|^{-n'/2}\cdot f'(g') f''(g'')
$$
on the left hand side. Here $|a|$ is the adelic norm of $a$.
The Hall product $f'*f''$ is then given by the parabolic pseudo-Eisenstein series
$$(f'*f'')(g)\,\, = \sum_{\gamma\in P_{n',n''}(\QQ) \backslash GL_n(\QQ)} f(\gamma g).
$$
Now, writing
$$\widetilde \CT_n(f'*f'') (g)\,\,= 
\int_{u\in U(\QQ)\backslash U(\Aen)}  \,\, \sum_{\gamma\in P_{n',n''}(\QQ) \backslash GL_n(\QQ)}
f(\gamma u g) \delta_n(g)^{1/2} du,$$
we notice that the Grassmannian $\operatorname{Gr}(n',\QQ^{n}) = P_{n',n''}(\QQ)\backslash GL_n(\QQ)$
splits, under the right $U(\QQ)$-action, into $n\choose n'$ orbits (Schubert cells) 
$$\Sigma_w = P_{n',n''}(\QQ)\backslash w U(\QQ),\quad  w\in Sh(n', n'').$$
Notice that for $w\in Sh(n', n'')$ we have $U_w = U\cap w^{-1} P_{n',n''}w$.
This means that we can write each $\gamma\in\Sigma_w$ uniquely in the form
$\gamma = P_{n'n''}(\QQ)\cdot w \cdot v$ for $v\in U_w(\QQ)\backslash U(\QQ)$ and so
\[
\begin{split}
\widetilde \CT_n(f'*f'')(g) \,\, =\sum_{w\in Sh(n',n'')} \int_{u\in U(\QQ)\backslash U(\Aen)} \sum_{v\in U_w(\QQ)
\backslash U(\QQ)} f(wvug) \delta_n(g)^{1/2}du \cr
=\sum_{w\in Sh(n',n'')} \int_{\widetilde u\in U_w(\QQ)\backslash U(\Aen)} 
f(w\widetilde u g) \delta_n(g)^{1/2}d\widetilde u,
\end{split}
\]
and we identify the integral over $\widetilde u$ corresponding to $w$, with
$M_w(\widetilde \CT_{n'}(f')\otimes\widetilde\CT_{n''}(f''))$.  Note that this argument shows, in particular,
that $M_w$ is indeed applicable in this case as the domain of integration reduces
to a compact one (since  all we did was re-partition the  integral for $\widetilde\CT_n(f'*f'')(g)$,
which was over  a compact
domain to begin with). We leave the rest to the reader. 

\vskip  .2cm

Let us note a version of the above statement for the constant term of the $n$-tuple Hall product.
The proof is similar. 

\begin{prop} \label{prop:CT-iterated}
Let $\varphi_1, ..., \varphi_n\in C^\infty_0(\RR_+)$ and
$\varphi = \varphi_1\otimes ...\otimes\varphi_n \in C^\infty_0(\RR_+^n)$.
Then
\[
\widetilde\CT_n(*_{1^n}(\varphi)) \,\,=\,\,\sum_{w\in\Sen_n} M_w(\varphi). \qed
\]

\end{prop}

\paragraph{D. The Mellin transform of the constant term.} 

For $f\in H_n$ we set
$Ch_n(f) = \Mc(\widetilde \CT_n(f))$.  

\begin{prop}\label{prop:Ch-convergence}
The Mellin integral for $Ch_n(f)$ converges to an analytic function in the region
$\CC^n_>$. 
 
\end{prop}

\begin{proof} The Mellin transform of  $\widetilde\CT_n(f)(a)=\delta_n(a)^{1/2}\CT_n(a)$
differs from $\Mc(\CT_n(a))$ by a shift of variables, and our statement is equivalent to saying that
$\Mc(\CT_n(a))$ converges for 
$$\RRe(s_1-s_2)>0, \,\,\RRe(s_2-s_3)>0,\,\, \cdots, \RRe(s_{n-1}-s_n)>0,\,\,\RRe(s_n)>0.$$
To see this, note that 
 by Proposition \ref{prop:CT-support} and of boundedness of $\CT_n(f)$,
the integral 
is bounded by 
\[
\begin{split} \operatorname{const}\int_{a_1=0}^c \int_{a_1a_2=0}^c \cdots \int_{a_1...a_n=0}^c
a_1^{s_1-s_2} (a_1a_2)^{s_2-s_3} \cdots (a_1...a_n)^{s_n}\times \cr
\times d^*a_1 d^*(a_1a_2) \cdots d^*(a_1...a_n).
 \end{split}
 \]
 Since $\int_0^c a^s d^*a$ converges for $\RRe(s)>0$, the claim follows. 
  \end{proof}
  
   \begin{prop}\label{prop:Ch-mer}    For any $f\in SH_n$, the function $Ch_n(f)$ extends to a 
   meromorphic function on $\CC^n$. 
  \end{prop}
  
  Before giving the proof, we recall the properties of a classical type of Eisenstein series
  due to Selberg \cite{selberg}. 
  
   For any $s\in\CC$ we denote by $\Een(s)$ the following function on $\Bunn_1$:
 \be\label{eq:Een(s)}
 \Een(s): \,E\longmapsto \deg(E)^s \,=\, \exp( s\cdot \ln(\deg(E))).
 \ee
 The (formal) Hall product
 \be\label{eq:Eisenstein-Selberg}
 \Een(s_1)*\cdots \Een (s_n) \,\,=\,\, *_{1^n}(a_1^{s_1}...a_n^{s_n})
 \ee
 is a series of functions on $\Bunn_n$, 
 known as the (primitive) {\em Eisenstein-Selberg series},
see \cite{selberg} and  \cite{jorgensen-lang} \S 8.3. 

\begin{prop}\label{prop:eisenstein-selberg}
 (a) The series \eqref{eq:Eisenstein-Selberg}
converges
  for $s=(s_1, ..., s_n)\in\CC^n_>$, 
    to a $C^\infty$-function on $\Bunn_n$.
  
  (b) For any $g\in\Bunn_n$ the 
    function $\bigl(\Een(s_1)*\cdots *\Een(s_n)\bigr)(g)$ 
  extends to a meromorphic function in the $s_i$,
  with position and order of poles independent on $g$. 
  
  (c)   The twisted constant term of $\bigl(\Een(s_1)*\cdots *\Een(s_n)\bigr)(g)$ 
  as a function on $g$ is given by  
  $$\widetilde\CT_n \bigl(\Een(s_1)*\cdots *\Een(s_n))(a_1, \dots, a_n\bigr) \,\,=\,\,\sum_{w\in\Sen_n}
  a_1^{s_{w(1)} }\cdots a_n^{s_{w(n)} }\prod_{i<j\atop w(i)>w(j)} \Phi(s_i-s_j).
  $$
\end{prop}  

\begin{proof} For (a), 
see, e.g., \cite{jorgensen-lang}, \S 8.5, Remark, and take into
    account the Ringel twist in the definition of $*$ which translates the shifts by $1/2$
    into shifts by 1.
        See also \cite{goldfeld}, Proposition 10.4.3 for a slightly weaker statement. 
        
        For (b),  see \cite{jorgensen-lang}, \S 8.6-7.
        
         Finally, (c) follows by 
          the formula   \eqref{eq:CT-Eis} applied to
         the function $a^s$, $s\in\CC^n_>$ (the application is legal because of
         the decay conditions) and then using Proposition 
         \ref{prop:M-w-a-s}. 
     \end{proof}

\noindent {\em  Proof of Proposition \ref{prop:Ch-mer}: } 
It is enough to assume that $f=f_1*\cdots * f_n$, where $f_\nu\in H_1=C^\infty_0(\RR_+)$. 
  Let $F_\nu=\Mc(f_\nu)\in\PW(\CC)$ be the Mellin transform of $f_\nu$.  
    Then $f_\nu=\Nc(F_\nu)$, and
  the inverse Mellin integral (understood as in Proposition \ref{prop:mellin-inversion})
   can be taken along any vertical line $\RRe(s)=\sigma_\nu$.
   
     Let us now choose $\sigma_1, ..., \sigma_n$ such that $\sigma_{\nu+1}-\sigma_{\nu}>1$ for each $\nu=1, ..., n-1$
     and $\sigma_n>1$.
  The equalities $\Nc(F_\nu)=f_\nu$ then imply that
    $$f(g)\,\,= {1\over(2\pi i)^n} \int_{\RRe(s_\nu)=\sigma_\nu}F_1(s_1)\cdots F_n(s_n)
   \bigl( \Een(-s_1)*\cdots *\Een(-s_n)\bigr)(g) ds_1\cdots ds_n.
    $$ 
    Substituting the formula for the twisted constant term of $\bigl( \Een(-s_1)*\cdots *\Een(-s_n)\bigr)(g)$
    from Proposition \ref{prop:eisenstein-selberg}(c) into the integral for $f(g)$, we represent
    $\widetilde\CT_n(f)$ as the inverse Mellin transform
  of the function
  \[
  F(s_1, ..., s_n) \,\,=\sum_{w\in\Sen_n} F_1(s_{w(1)})\cdots F_n(s_{w(n)}) \prod_{i<j\atop w(i)>w(j)} \Phi(s_j-s_i),
  \]
  which is analytic in the region $\RRe(s_{\nu+1})-\RRe(s_\nu)>1$. Further, if we take $\sigma_1, ..., \sigma_n$
  such that $\sigma_{\nu+1}-\sigma_\nu>1$, $\sigma_n>1$, then $F$ is bounded on the vertical subspace 
  $\RRe(s_\nu)=s_\nu$.
  Indeed, each $F_i$, being a Paley-Wiener function, decays exponentially at the imaginary infinity.
  On the other hand, the lemma below shows that   $\Phi(s)$   is  bounded  on
  vertical lines $\RRe(s)=\sigma_0>1$. Therefore we can apply the Mellin inversion
  (Proposition \ref{prop:mellin-inversion}) to $F$ and obtain that
  $Ch_n(f)=\Mc(\widetilde\CT_n(f))=F(s_1, ..., s_n) $ and so it is meromorphic.  
  
  \begin{lemma} For every $\sigma_0 > 1$, the function $\Phi(\sigma_0+it)$
is bounded, as a function of $t\in\RR$, and decays as $|t|\to\infty$.
\end{lemma}

\begin{proof}
Indeed, for $s=\sigma_0+it$, $\sigma_0 >1$ we have 
\[
{\zeta(s)/\zeta(s+1)}\,\, =\,\, \sum_{n=1}^\infty {\varphi(n) n^{-s-1}},
\]
where $\varphi(n)=|(\ZZ/n)^\times|$ is the Euler function.  This is bounded by
 \[\sum n\cdot n^{-\sigma_0-1}=\zeta(\sigma_0).
 \]
 Further, $\Gamma({s\over 2})/\Gamma({s+1\over2})$ decays at infinity as $s^{-1/2}$, as it follows from the Stirling formula. 
\end{proof}

\paragraph{ E. Intertwiners and the constant term.} We now study the action of the intertwiners $M_w$ on
the Mellin transform of the constant term.  

\begin{prop}\label{prop-mellin-M-w}
 For $\varphi\in C^\infty_0(\RR_+^n)$ and any $w\in\Sen_n$  we have
 \[
 \Mc(M_w(\varphi))(s) \,\,=\,\,\Mc(\varphi)(w(s))\cdot \Phi_w(s). 
 \]
\end{prop}

\begin{proof} Write $\varphi$ as the inverse Mellin integral of a Paley-Wiener function $F$ over any
vertical subspace
$\sigma + i\RR^n$ inside $\CC^n_>$, and apply Proposition \ref {prop:M-w-a-s}. 
\end{proof}

At this point, we can finish the proof of Theorem \ref{thm:main}.  It remains only  to prove that $Ch$ is a homomorphism of
algebras, i.e., that
\be\label{eq:ch-homomorphism}
Ch_n(f'*f'')\,=\, Ch_{n'}(f')\sm Ch_{n''}(f''), \quad n=n'+n''
\ee
for any $f'\in SH_{n'}$ and $f''\in SH_{n''}$. Using the formula 
\eqref{eq:CT-Eis} for the left hand side and the definition of the shuffle product $\sm$ for the right
hand side, we write this as an equality of two sums over shuffles
\be\label{eq:sum-shuffles-putative}
\begin{gathered}
\sum_{w\in Sh_{n',n''}} \Mc\bigl(M_w(\widetilde\CT_{n'}(f')\otimes\widetilde\CT_{n''}(f''))\bigr) (s) \,
=
\cr
 =\sum_{w\in Sh_{n',n''}} \Mc\bigl( 
\widetilde\CT_{n'}(f')\otimes\widetilde\CT_{n''}(f'')
\bigr) (w(s))\cdot \Phi_w(s).
\end{gathered}
\ee
As $f', f''$ belong to the subalgebra $SH$, we can write them as
\[
f'=*_{1^n}(\varphi'), \quad f''=*_{1^n}(\varphi'')
\]
for some $\varphi'\in C^\infty_0(\RR_+^{n'})$, $\varphi''\in C^\infty_0(\RR_+^{n''})$.
By Proposition \ref {prop:CT-iterated}, we have
\[
\widetilde\CT_{n'}(f') \,\,=\,\,\sum_{w'\in\Sen_{n'}} M_{w'}(\varphi'),
\]
and similarly for $\widetilde\CT_{n'}(f')$.  Substituting this to the LHS of the putative equality 
\eqref{eq:sum-shuffles-putative}, we find that it is equal to
\be\label{eq:5.23}
\sum_{w\in \Sen_n} \Mc(M_w(\varphi))\,\, \buildrel \ref{prop-mellin-M-w}\over =\,\,
\sum_{w\in\Sen_n} \Mc(\varphi)(w(s))\cdot\Phi_w(s),
 \quad \varphi = \varphi'\otimes\varphi''. 
\ee
On the other hand, writing $s\in\CC^n$ as $(s', s'')$ with $s'\in\CC^{n'}, s''\in\CC^{n''}$,
we have
\[
\Mc\bigl(\widetilde\CT_{n'}(f')\otimes\widetilde\CT_{n''}(f'')\bigr) (s) \,\,=\,\,
\Mc(\widetilde\CT_{n'}(f'))(s')\cdot \Mc(\widetilde\CT_{n''}(f''))(s''),
\]
and so the summand in the RHS of \eqref{eq:sum-shuffles-putative}
corresponding to $w\in Sh_{n',n''}$, is equal by Proposition 
\ref{prop:CT-iterated}, to $\Phi_w(s)$ times
\[
\begin{gathered}
\sum_{w'\in\Sen_{n'}\atop w''\in\Sen_{n''}} \Mc(\varphi')(w'(s'))\cdot \Mc(\varphi'')(w''(s''))
\cdot \Phi_{w'}(s')\Phi_{w''}(s'')
\,\,=\cr
=\sum_{w'\in\Sen_{n'}\atop w''\in\Sen_{n''}} \Mc_{w'\times w''}(\varphi)((w'\times w'')(s))\cdot
\Phi_{w'\times w''}(s),
\end{gathered}
\]
and further summation over $w$ gives the same result as 
\eqref{eq:5.23}. \qed

 \vskip 1cm
   
   \section{Quadratic relations and Eisenstein series.}\label{sec:quadration-relations}
   
   Let 
   \[
   S=\bigoplus_{n=0}^\infty S_n, \quad S_0=\CC,
   \]
   be a graded associative algebra over $\CC$.
The space of degree $n$ relations among elements of degree $1$ is then
\be\label{eq:R-n}
R_n\,\,=\,\,\Ker\{ S_1^{\otimes n} \lra S_n\}\,\,\subset\,\, S_1^{\otimes n}.
\ee
Here we are interested in quadratic relations ($n=2$)  for the algebra $SH$ generated
by $SH_1=H_1=C^\infty_0(\RR_+)$.  Because of the analytic nature of elements of
$H$ it is not reasonable to look for relations inside the algebraic tensor product
$H_1\otimes H_1$ and we consider a  completion of it, namely the space
\[
H_1\widehat\otimes H_1\,\,:=\,\, \Dc(\RR_+^2)_{\on{abs}}
\]
of absolutely tempered distributions on $\RR_+^2$, see
Corollary \ref{cor:abs-tempered}.

 \begin{prop}
  If $f\in H_1\widehat\otimes H_1$, then the series
 $$\widehat *_{1,1}(f)(E) \,\,=\,\,\sum_{E'\subset E} 
  \deg(E')^{1/2} \deg(E/E')^{-1/2} f(\deg(E'), \deg(E/E')), \quad E\in\Bunn_2,$$
  converges absolutely, defining a distribution $\widehat *_{1,1}(f)$ on $\Bunn_2$.
  The resulting linear map $\widehat *_{1,1}: H_1\widehat\otimes H_1\to \Dist(\Bunn_2)$
  extends the Hall mltiplication $*_{1,1}: H_1\otimes H_1\to H_2$.
\end{prop}

\begin{proof} 
The points $(\alpha,\beta)=(\deg(E'), \deg(E/E'))$ lie on the hyperbola 
$\alpha\beta=\deg(E)$. An absolutely tempered distribution
decays exponentially at the infinity of $\RR_+^2$,  in particular
at the infinity of any such hyperbola. Now
the number of subbundles in $E=(L,V,q)$ of given degree $\alpha= 1/a$
is one half the number of primitive vectors in $L$ of norm $a$. 
This number of all lattice vectors of norm $a$
 grows linearly with $a$, so exponential decay of $f$ ensures the convergence.
 \end{proof}
 
 \begin{rem}
 It is possible that one can extend  $H$  to a bigger algebra,
 consisting of some analogs of
 absolutely tempered distributions on the $\Bunn_n$,
 which have sufficient decay at the infinity.  Note that
 the concept of a tempered distribution on a semisimple Lie group was introduced
 by Harish-Chandra \cite{harish-chandra}.
  \end{rem}
 
We will therefore understand quadratic relations in $SH$ is a wider sense, as
elements of the space
\be
\widehat R_2\,\,=\,\, \Ker(\widehat *_{1,1}) \,\,\subset\,\,   H_1\widehat\otimes H_1.
\ee
Let also $\Rc_2$ be 
the space of entire functions
  $F\in\Oc(\CC^2)_{\on{pol}}$ such that
  \be\label{eq:script-R-2}
  F(s_1, s_2) \,+\, \Phi(s_1-s_2) F(s_2, s_1)\,\,=\,\,0.
  \ee

  \begin{prop}\label{prop:quadration-relations-properties}
     The Mellin transform identifies $\widehat R_2$ with $\Rc_2$. 
  
  \end{prop}

 \begin{proof} This follows from
   an instance of Eq. \eqref{eq:ch-homomorphism} for $m=n=1$ but applied
  to absolutely tempered distributions instead of functions with compact
  support. The proof in the new case is the same, given the decay (to define the Hall product)
  and the analyticity of the Mellin transform. 
  \end{proof}

  Note that $\Rc_2$ is a module over the ring $\Oc(\CC^2)^{\Sen_2}_{\on{pol}}$ of
  symmetric entire functions of polynomial growth on vertical planes. 
  
  \begin{ex}\label{ex:zeta-in-R-2}
  Let $P(s) = s(s-1)(s+1)$. Then the function
  \[
  F_{1,1}(s_1, s_2) \,=\, P(s_1-s_2) \zeta^*(s_1-s_2)
  \]
  belongs to $\Rc_2$. Further, for any $\lambda_1, \lambda_2\in\RR_+$
  the function
  \[
  F_{\lambda_1,\lambda_2}(s_1, s_2) \,\,=\,\, (\lambda_1^{s_1}\lambda_2^{s_2} + \lambda_1^{s_2}
  \lambda_2^{s_1}) F_{1,1}(s_1, s_2)
  \]
  again lies in $\Rc_2$ by the remark above. Let
   \[
   \nabla_a \,\,=\,\,\ P\bigl( a{d\over da}\bigr) \,\,=\,\,a^3{d^3\over da^3} - a^2{d^2\over da^2}.
   \]
   The inverse Mellin transform of $F_{1,1}$ is, in virtue of Proposition 
  \ref{prop:mellin-delta-function}  and the Riemann formula \eqref{eq:zeta-theta}, equal to
  \[
  \Psi_{1,1}(a_1, a_2) \,\,=\,\,\delta_1(a_1a_2)\cdot \nabla_{a_1}\theta(a_1^2)\,\,\in\,\,\widehat R_2,
  \]
  and the inverse Mellin transform of $F_{\lambda_1, \lambda_2}$ is the distribution
  \[
  \Psi_{\lambda_1, \lambda_2}(a_1, a_2) = \Psi_{1,1}(a_1/\lambda_1, a_2/\lambda_2)\,+\,
  \Psi_{1,1}(a_1/\lambda_2, a_2/\lambda_1)\,\,\in\,\, \widehat R_2. 
  \]
  This gives a 2-parameter family of quadratic relations in $SH$. 
   \end{ex}
   
   \vskip .2cm

  \begin{rem} This 2-parameter family of relations is analogous to the family of relations
   $$[\Oc(m+1)]*[\Oc(n)] -q[\Oc(n)]*[\Oc(m+1)] \,\,=\,\, q[\Oc(m)]*[\Oc(n+1)] -
   [\Oc(n+1)]*[\Oc(m)]$$
   in the Hall algebra of the category of vector bundles on $\PP^1_{\FF_q}$, see \cite{kapranov} \S 5.2
   or \cite{baumann-kassel}, Lemma 16. 
    \end{rem}
    
    We now explain the relation of the above quadratic relations with the functional equation for
    {\em Eisenstein-Maas series}
     $$\Eb(\tau,s) \,\,=\,\,{1\over 2} \sum_{(m,n)=1} {\Im(\tau)^{s}\over |m+n\tau|^{2s}}, \quad \tau\in\HH, \,\, \RRe(s)>1,$$
 see \cite{goldfeld}, \S 3.1.
   It is classical  that $\Eb(\tau, s)$ extends to
 a function meromorphic in the entire $s$-plane and satisfying the functional equation
 $$\Eb(\tau, s) \,\,=\,\,{\zeta^*(2s-1)\over\zeta^*(2s)} \,\Eb(\tau, 1-s).$$
 Further, the poles of $\Eb(\tau, s)$ are all among the poles of the ratio of
 the $\zeta^*$-functions, in particular, they do not depend on $\tau$.

   \vskip .2cm
   
   On the other hand,
  recall \eqref{eq:Een(s)} the function 
  \[
  \Een(t): \Bunn_1\to \CC, \quad E\longmapsto \deg(E)^t.
  \]   
  Here $t\in\CC$ is a fixed complex number. 
  This function does not lie in $H_1=C^\infty_0(\RR)$. Nevertheless, the correspondence
  $t\mapsto \Een(t)$ can be seen as a kind of $H$-valued distribution
  (``operator field") on $\CC$ (or, rather, on $i\RR\subset\CC$). That is, for any
  Paley-Wiener function $G(t)$ we have a well defined
  element
  \[
  \int_{i\RR}  \Een(t) G(t) dt \,\,\in\,\,  H_1.
  \]
  This simply the function $E\mapsto f(\deg(E)^{-1})$, where $f=\Nc(G)\in C^\infty_0(\RR_+)$. 
  
  \vskip ,2cm
  
  Proposition \ref{prop:eisenstein-selberg}(a) 
  implies that for $\RRe(t_1-t_2)>0$ the Hall product
  $\Een(t_1)* \Een(t_2)$ defined as a formal series,
  converges to a real analytic function on $\Bunn_2$. 
   This function
  essentially reduces to the series $\Eb(\tau, s)$
    above. 
  Indeed, let $E_\tau$ be the bundle of rank 2 and degree 1 corresponding to
 $\tau$ as in Example \ref{ex:Bun-2-1}. Rank 1 subbundles $E'=E'_{m,n}$ in $E_\tau$ are parametrized
 by pairs $(m,n)\in\ZZ^2$ of coprime integers, taken modulo simultaneous change of sign. Explicitly,
 the primitive sublattice $L'_{m,n}$ of $E'_{m,n}$ is spanned by $m+n\tau$, and we have
 $$\deg(E'_{m,n}) = {\Im(\tau)^{1/2}\over |m+n\tau|},  \quad \deg(E_\tau/E'_{m,n})= {|m+n\tau|\over \Im(\tau)^{1/2}}.
 $$
 Therefore
 \be\label{eq:eisenstein-ringel-shift}
 (\Een(t_1)*\Een(t_2))(E_\tau) \,\,=\,\, \Eb (\tau, (t_1-t_2+1)/2).
 \ee
  This means that the product $\Een(t_1)*\Een(t_2)$   extends to a meromorphic function
of $t_1, t_2$ (with values in the space of functions on $\Bunn_2$)
and we  can write a formula looking like ``quadratic commutation relations" in $H$:
\be\label{eq:zeta-comm-relations}
\Een(t_1)*\Een(t_2) \,\,-\,\, \Phi(t_1-t_2) \,\Een(t_2)*\Een(t_1)\,\,=\,\,0.
\ee
 The
two summands in   \eqref{eq:zeta-comm-relations}  are given by series converging in different regions,
having no points in common, and the relations should be understood via
analytic continuation.  This way of understanding commutation relations is  quite standard
in  the theory of vertex operators \cite{FLM}.  In our situation it is modified as follows.

In order to translate the relations \eqref{eq:zeta-comm-relations}  into actual elements of $\widehat R_2$,
we can rewrite them in the form ``free of denominators"
\be \label{eq:free-denominators}
*_{1,1} \bigl\{P(t_1-t_2) \cdot \zeta^*(t_1-t_2+1)\cdot
  a_1^{t_1}a_2^{t_2} \,\,-\,\, P(t_1-t_2) \cdot\zeta^*(t_1-t_2)
\cdot a_1^{t_2}a_2^{t_1} \bigr\} \,\,=\,\,0. 
 \ee
Here we write $\Een(t_1)\otimes \Een(t_2)$ as the function 
 $(a_1, a_2)\mapsto a_1^{t_1}a_2^{t_2}$ on 
 $\Bunn_1\times\Bunn_1=\RR_+^2$. 
 We then ``compare coefficients" in both sides of this equality at any
 $\lambda_1^{t_1}\lambda_2^{t_2}$, $\lambda_\nu\in\RR_+$, by multiplying
 with $\lambda_1^{-t_1}\lambda_2^{-t_2}$ and integrating 
 (performing the inverse Fourier-Schwartz transform) along
 any vertical 2-plane, which we can choose separately for each summand. 
 This gives a
 family of distributions $\Psi_{\lambda_1, \lambda_2}(a_1, a_2)\in \widehat R_2$ 
 which is the same as in Example
 \ref{ex:zeta-in-R-2}. 
 
 \vskip .2cm
   
We can thus say that quadratic relations  such as \eqref{eq:zeta-comm-relations}
are  built into the very definition of the shuffle algebra.

 \vskip 1cm

 \section{Wheels, cubic relations,   and zeta roots.}
 
 \paragraph{A. Wheels.}
   Let $\lambda(s)$ be a meromorphic function on $\CC$ with a simple pole at $s=0$ and
no other singularities.
In this section we sketch a general approach to higher order relations in
the symmetric shuffle algebra
$\SShuff(\lambda)$ and illustrate it on the case of cubic relations in the
shuffle algebra completion of the spherical Hall algebra $SH$, which corresponds to
\[ \lambda(s)\,\,=\,\, \Lambda(s)\,\, =\,\, \zeta^*(-s)(s-1)(-s-1).
 \]
Our approach is based on studying the following additive patterns of roots of
$\lambda$ which were introduced  in \cite{FJMM} and used 
 in the case when $\lambda$ is rational.
 
 \begin{df} A {\em wheel}  of length $n$ for $\lambda$ is a sequence $(s_1, ..., s_n)$
 of distinct complex numbers such that
 \[
 \lambda(s_2-s_1)=0, \,\,\lambda(s_3-s_2)=0, \,\, \cdots,  \lambda(s_n-s_{n-1})=0,\,\,
 \lambda(s_1-s_n)=0.
 \]
  Wheels $(s_1, ..., s_n)$  and  $(s_1+c, ..., s_n+c)$  for
 $c\in \CC$, will be called {\em equivalent}.  
 \end{df}
 In other words, equivalence classes of wheels are the same as 
 ordered sequences
 \[
 (z_1, ..., z_n)\in(\CC^*)^n, \quad \lambda(z_i)=0, \quad \sum_{i=1}^n z_i=0, \quad
 \sum_{i=p}^q z_i \neq 0, \quad (p,q)\neq (1,n).
 \]
  
 \begin{ex}
 All wheels for $\Lambda(s)$ have length 3 or more. The  sequences
 corresponding to wheels of length 3 have, up
 to permutation, the form
 \[
 (z_1, z_2, z_3) \,\,=\,\, (\rho, 1-\rho, -1),
 \]
 where $\rho$ runs over nontrivial zeroes of $\zeta(s)$. 
 Indeed, zeroes of $\Lambda$ are of the form $s=\rho$ together with one more zero $s=-1$. 
 So there are no pairs of them summing up to 0 and the only triples  summing to 
 up 0 are as stated.
 \end{ex}

 \paragraph{B. Relations and bar-complexes.}
  Let $S$ be a graded associative algebra as in \S \ref{sec:quadration-relations}. A systematic way of
 approaching relations in $S$ is via the bar-complexes
  $$B_n^\bullet = B_n^\bullet(S) \,\,=\,\,\biggl\{S_1^{\otimes n}\to
\cdots \to \bigoplus_{i+j+k=n} S_i\otimes S_j\otimes S_k\to \bigoplus_{i+j=n} S_i\otimes S_j \to S_n\biggr\}.
$$
  Here $i,j,k,...$ run over positive integers. The grading is such that 
$S_1^{\otimes n}$ is in degree $(-n)$, while $S_n$ is in degree $(-1)$.
 The differential is given by
$$d(s_1\otimes \cdots s_p) \,\, =\sum_{i=1}^{p-1} (-1)^{i-1} s_1\otimes ...\otimes s_{i-1}\otimes s_{i}s_{i+1}\otimes s_{i+2}\otimes
...\otimes s_p,$$
so that the condition $d^2=0$ follows from the associativity of $S$. It is well known  that
$$H^{-i}(B^\bullet_n(S)) \,\,=\,\,\on{Tor}_i^S(\CC, \CC)_n,$$
the part of the Tor-group which has degree $n$ w.r.t. the grading induced from that on $S$. 
In particular, the rightmost cohomology has the meaning of the space of generators in degree $n$,
and the previous one is interpreted as the space  of  relations  which have degree $n$ with respect to
the grading on the generators (which,  a priori, can be  present in any degree). 

\vskip .2cm

As in \eqref{eq:R-n}, let $R_n$ be the space of degree $n$ relations among generators in degree 1.
For instance, quadratic relations are found as  $R_2=H^{-2}(B_2^\bullet)$.
The next case of  cubic relations corresponds to the complex
$$
 B_3^\bullet \,\,=
\,\,\bigl\{ S_1\otimes S_1\otimes S_1, 
 \buildrel d_{-3}\over\lra (S_2\otimes S_1)\oplus (S_1\otimes S_2)
\buildrel d_{-2}\over\lra S_3\bigr\}.
$$
  We treat this case directly. Denote
$$R_{12}=R_2\otimes S_1, \,\,\, R_{23}=S_1\otimes R_2\quad\subset\quad S_1\otimes S_1\otimes S_1.$$
We have then an inclusion
$R_{12}+R_{23}  \,\,\subset \,\,R_3
$
of subspaces in $S_1^{\otimes 3}$. The left hand side of this inclusion
is, by definition, the space of those cubic relations which follow algebraically from the quadratic ones.
Thus the quotient 
$$R_3^{\operatorname{new}}\,\,=\,\, R_3/(R_{12}+R_{23})$$
can be seen as the space of  ``new", essentially cubic, relations. 

\begin{prop}\label{prop:cubic-relations-bar}
 Assume that the multiplication map $S_1\otimes S_1\to S_2$ is surjective.
Then   $R_3^{\operatorname{new}}$ is identified with $H^{-2}(B_3^\bullet)$,
the middle cohomology space of $B_3^\bullet$. 
\end{prop}

\begin{proof} Denote for short
$$V=S_1^{\otimes 3}, \,\,\, A=R_{12}, \,\,\, B=R_{23},\,\,\, C=R_3,$$
so that $A,B\subset C\subset V$. Under our assumption, the complex $B_3^\bullet$ 
can be written as
$$V\buildrel \delta_{-3}\over\lra (V/A)\oplus (V/B) \buildrel \delta_{-2}\over\lra V/C,$$
with $\delta_{-3}$ being the difference of the two projections, and 
$\delta_{-2}$ being the sum of the two projections. It is a general fact that in such
a situation the middle cohomology is identified with $C/(A+B)$. Explicitly, if
$(v+A, w+B)\in\Ker(\delta_{-2})$, then $v+w\in C$. The image of $v+w$ in $C/(A+B)$
depends only on the class of $(v+A, w+B)$ in $\Ker(\delta_{-2})/\Im(\delta_{-3})$. 
We leave the rest to the reader.
\end{proof}

\paragraph{C. Localization of the bar-complexes.}
We now apply the above to the two graded algebras
\[
\SShuff(\lambda) \,\,\,\,\subset\,\, \,\, \Sc \,:=\,\biggl( \bigoplus_n \Oc(\CC^n)^{\Sen_n}, \,\, \star\biggr).
\]
By definition, these algebras coincide in degrees 0 and 1, and $\SShuff(\lambda)$ is
the subalgebra in $\Sc$ generated by the degree 1 part which is $\Sc_1=\Oc(\CC)$. 
Accordingly, the space of relations of any degree $n$ among degree 1 generators
in $\Sc$ and $\SShuff(\lambda)$ are the same. As in \S \ref{sec:quadration-relations},
we will look at relations   as elements of the completed tensor
product. That is, for any two Stein manifolds $M$ and $N$ we write
\[
\Oc(M)\widehat\otimes\Oc(N) \,\, :=\,\,\Oc(M\times N)
\]
and understand $\Sc_1^{\widehat \otimes n}=\Oc(\CC^n)$ accordingly. 
The version of the bar-complex of $\Sc$ using $\widehat\otimes$,
has the form
\[
\Bb_n^\bullet=
\biggl\{\Oc(\CC^n) \to \cdots \to
\bigoplus_{i+j+k=n}\Oc(\CC^n)^{\Sen_i\times\Sen_j\times\Sen_k}
 \to \bigoplus_{i+j=n} \Oc(\CC^n)^{\Sen_i\times\Sen_j} \to \Oc(\CC^n)^{\Sen_n}\biggr\}
\]
Notice that each term of this complex is a module over the ring $\Oc(\CC^n)^{\Sen_n}$ of symmetric
entire functions, and the differentials, coming from multiplication in $\Sc$, are
$\Oc(\CC^n)^{\Sen_n}$-linear. This means that $\Bb_n^\bullet$ is the complex of global section
of a complex of vector bundles $\Bc_n^\bullet$ on the Stein manifold $\Sym^n(\CC)$.
Explicitly, for $i_1+...+i_p=n$ we denote by
\[
\pi_{i_1, ..., i_p}: \Sym^{i_1}(\CC) \times \cdots\Sym^{i_p}(\CC) \lra\Sym^n(\CC)
\]
the symmetrization map (a finite flat morphism). Then
\be\label{eq:B-as-direct-image}
\Bc_n^{-p} \,\,=\,\,\bigoplus_{i_1+...+i_p=n} (\pi_{i_1, ..., i_p})_* \,\,\Oc_{\prod \Sym^{i_\nu}(\CC)},
\ee
in particular, $\Bc_n^\bullet$ is a complex of holomorphic vector bundles on $\Sym^n(\CC)$. 
 This allows us to approach the cohomology of $\Bb_n^\bullet$ (and, in particular, relations in $\Sc$)
 in a more geometric way, by studying the cohomology of the fibers 
 \[
 \Bc^\bullet_{n, T} \,\,=\,\,\Bc^\bullet_n \otimes_{\Oc_{\Sym^n(\CC)}} \Oc_T
 \]
  of the complex
 $\Bc_n^\bullet$ over various points $T\in\Sym^n(\CC)$. Now, our main technical result
 is as follows.
  
   \begin{thm}\label{thm:bar}
 Let $T=\{s_1^0, ..., s_n^0\}\in\Sym^n(\CC)$ be an unordered collection of distinct points. Suppose that
 no subset of $T$ (in any order) is a wheel. Then $ \Bc^\bullet_{n, T}$ is exact everywhere except
 the leftmost term, where the cohomology is 1-dimensional. 
 \end{thm}
 
 Recall that similar exactness of all the bar-complexes $B_n^\bullet(S)$ for a graded algebra $S$
 means that $S$ is quadratic Koszul. The wheels represent therefore local obstructions to
 Koszulity for
 $\Sc$. 
 
 \paragraph{D. Cubic relations in $SH$ and zeta roots.}
 Before giving the proof of Theorem \ref{thm:bar}, let us explain how to apply it to the case of cubic relations
  for $\lambda=\Lambda$.
 Let $\rho$ be a nontrivial zero of $\zeta(s)$. Denote by $W_\rho\subset \Sym^3(\CC)$
the subset of points $\{s_1, s_2, s_3\}$ such that, after some renumbering of the $s_i$
we have $s_2-s_1=\rho$,  $s_3-s_2=1-\rho$ (such a renumbering is then unique).
 Let $W$ be the union of the $W_\rho$ over
all  nontrivial zeroes $\rho$ of $\zeta(s)$. The following is then straightforward.

\begin{prop}\label{prop-Z-rho}
 (a) Each $W_\rho$ is a complex submanifold in $\Sym^3(\CC)$, isomorphic to $\CC$,
the symmetric function $s_1+s_2+s_3$ establishing an isomorphism. 

(b) For $\rho\neq\rho'$ we have $W_\rho\cap W_{\rho'}=\varnothing$. 

(c) A point $\{s_1, s_2, s_3\}\in\Sym^3(\CC)$ lies in $W$, if and only if  it
is a wheel (in some numbering). \qed
\end{prop}

\begin{thm}\label{thm:cubic-relations-zeta}
 Let $\lambda(s)=\Lambda(s)$. 

(a) The multiplication map $\Sc_1\widehat\otimes\Sc_1\to \Sc_2$ is surjective, so,
by Proposition \ref{prop:cubic-relations-bar}, the space
\[
H^{-2}(\Bb_3^\bullet) \,\,=\,\, H^0(\Sym^3(\CC), \underline H^{-2}(\Bc_3^\bullet))
\]
is identified with the space of new cubic relations in $\Sc$ as well as in
 in $\SShuff(\Lambda)$. 

(b) The support of the coherent sheaf $\underline H^{-2}(\Bc_3^\bullet)$ is equal to $W=\bigsqcup W_\rho$.
If $\rho$ is a simple root of $\zeta(s)$, then $\underline H^{-2}(\Bc_3^\bullet)\simeq \Oc_{W_\rho}$
in a neighborhood of $W_\rho$. 
\end{thm}

\begin{rem}
 From the point of view of this section, a cubic relation in $\SShuff$ is
an entire function $F(s_1, s_2, s_3)\in\Oc(\CC^3)=\Sc_1^{\widehat\otimes 3}$
mapped to the zero element of $\Sc_3$ by the symmetric shuffle multiplication.
On the other hand, from the more immediate point of view of \S \ref{sec:quadration-relations},
a cubic relation in the spherical Hall algebra $SH$ is a distribution $f(a_1, a_2, a_3)$
on $\RR_+^3=(\Bunn_1)^3$, mapped to the zero distribution on $\Bunn_3$
by the Hall multiplication.  The relation between $f$ and $F$
is that of the Mellin transform.  Note that 
 whenever $f(a_1, a_2, a_3)$ is a relation, then so is
 the rescaling $f(\alpha a_1, \alpha a_2, \alpha a_3)$
for any $\alpha\in\RR_+$. Taking a weighted average of such rescalings, i.e.,
a convolution
 \[
 \int_0^\infty f(\alpha a_1, \alpha a_2, \alpha a_3)\varphi(\alpha) d^*\alpha
 \]
corresponds, on the Mellin transform side, to multiplying $F(s_1, s_2, s_3)$ by 
a function of the form $\psi(s_1+s_2+s_3)$. Since $s_1+s_2+s_3$ is a global coordinate
on each $W_\rho$, Theorem \ref{thm:cubic-relations-zeta}  admits the following
striking interpretation: {\em the space of  new cubic relations in $SH$ modulo
rescaling is identified with the space spanned by nontrivial zeroes of $\zeta(s)$. }

\vskip .2cm

 This fact is also true (with a similar proof)  for the Hall algebras corresponding to arbitrary 
compactified arithmetic curves
( = spectra of  rings of integers in number fields) as well as
(with an easier,  more algebraic proof)  for Hall algebras of smooth projective
curves  $X/\FF_q$.  
Note that for $X=\PP^1$ there are no new cubic relations \cite{kapranov, baumann-kassel},
while for  $X$ elliptic, new cubic relations were found 
 in \cite{schiffmann-drinfeld}.  Our results show that presence of cubic
 relations is a general phenomenon,
 holding for all curves $X/\FF_q$ of  genus $\geqslant 1$.
 
 We will give a detailed proof of Theorem \ref{thm:bar} and
 a sketch of proof of Theorem \ref{thm:cubic-relations-zeta},
 which will be taken up and generalized in a subsequent paper.
 
\end{rem}

\paragraph {E. Permuhohedra and the proof of Theorem \ref{thm:bar}.}
Our approach, similar to that of \cite{baranovsky, merkulov}, uses the {\em permutohedron},
which is the convex polytope
\[
P_n\,\,=\,\,\on{Conv}\bigl(\Sen_n \cdot (1,2,..., n)\bigr) \,\,\subset\,\, \RR^n
\]
of dimension $(n-1)$. Thus vertices of $P_n$ are the $n!$ vectors $(i_1, ..., i_n)$ for all
the permutations. It is  well known that faces of $P_n$ are in bijection with sequences
$(I_1, ..., I_p)$ of subsets of $\{1, ..., n\}$ which form a disjoint decomposition. We denote
$[I_1, ..., I_p]$ the case corresponding to $(I_1, ..., I_p)$. Subfaces of 
$[I_1, ..., I_p]$ correspond to sequences obtained by {\em refining} $(I_1, ..., I_p)$,
i.e., by replacing each $I_\nu$, in its turn, by a sequence $(J_{\nu, 1}, ..., J_{\nu, q_\nu})$
of subsets of $I_\nu$ forming a disjoint decomposition. 
Thus, as a polytope,
\[
[I_1, ..., I_p] \,\,\simeq \,\, P_{|I_1|}\times\cdot\times P_{|I_p|}, \quad
\dim [I_1, ..., I_p] \,=\, n-p.
\]
Let $C^\bullet(P_n)$ be the cochain comlplex of $P_n$ with complex coefficients. 
The basis of $C^m(P_n)$ is formed by  the $\1_F$, the characteristic functions of the
$m$-dimensional faces. We choose an orientation for each face. Then 
\[
d(\1_F)\,\,=\,\,\sum_{F'\supset F}\,\varepsilon_{FF'} \cdot \1_{F'}.
\] 
Here the sum is over $(m+1)$-dimensional faces $F'$ containing $F$, and 
$\varepsilon_{FF'}=\pm 1$ is the sign factor read from the orientations of $F$ and $F'$.

\vskip .2cm

On the other hand, \eqref{eq:B-as-direct-image} gives
a natural basis of $\Bc^{n-1-m}_{n,T}$ labeled
by the disjoint union of the preimages
\[
\pi_{i_1, ..., i_p}^{-1} (\{s_1^0, ..., s_n^0\}),\quad i_1+\cdots i_p=n.
\]
For a subset $I\subset \{1, ..., n\}$ let $T_I=\{s^0_i| i\in I \}\subset T$. 
Elements of each $\pi_{i_1, ..., i_p}^{-1} (\{s_1^0, ..., s_n^0\})$ are precisely
the
\[
(T_{I_1}, \cdots,  T_{I_p}) \,\,\in \,\,\Sym^{i_1}(\CC)\times\cdots\times\Sym^{i_p}(\CC)
\]
for all sequences of subsets $(I_1, ..., I_p)$, forming a disjoint decomposition of
$\{1, ..., n\}$. Denoting by $e_{I_1, ..., I_p}$ the corresponding basis vector
in $\Bc^{n-1-m}_{n, T}$, we get an isomorphism of graded vector spaces
\be\label{eq:B-to-C}
\Bc^\bullet_{n, T}\buildrel\sim
\over\lra C^\bullet(P_n)[n], \quad e_{I_1, ..., I_p} \mapsto \1_{[I_1, ..., I_p]}. 
\ee
To see the differential in $\Bc^\bullet_{n,T}$ from this point of view, consider the matrix
\[
\Lamed \,\,=\,\,\|\lambda_{ij}\|, \quad \lambda_{ij}=\lambda(s^0_i-s^0_j), \,\,\, 1\leqslant i,j\leqslant n, \,\, i\neq j. 
\]
Let $F\subset F'$ be a codimension 1 embedding of faces of $P_n$. That is,
$F'=[I_1, ..., I_p]$ and $F$ is   a minimal refinement of $F'$, i.e.,
is obtained by 
replacing some $I_\nu$ by $(I', I'')$ where $I', I''$
are nonempty sets forming a disjoint
decomposition of $I_\nu$. We put
\[
\lambda_{FF'} \,\,=\,\,\prod_{i'\in I'\atop i''\in I''} \lambda_{i'i''}. 
\]
It is immediately so see that the $\lambda_{FF'}$ satisfy the multiplicativity property
for any pair of composable codimension 1 embeddings:
\[
\lambda_{FF'}\lambda_{F'F''}\,\,=\,\,\lambda_{FF'}, \quad F\subset F'\subset F''.
\]
This implies that by putting
\[
d_\Lamed (\1_F) \,\,=\,\,\sum_{F'\supset F} \lambda_{FF'}\cdot \varepsilon_{FF'}\cdot \1_{F'},
\]
we obtain a differential $d_\Lamed$ in $C^\bullet(P_n, \CC)$ with square 0.
This is a certain perturbation of the cochain differential for $P_n$. 
We then see easily:

\begin{prop}\label{prop:B-to-C}
The isomorphism \eqref{eq:B-to-C} defines an isomorphism of complexes
\[
B^\bullet_{n,T}\lra \bigl(C^\bullet(P_n), d_\Lamed\bigr)[n]. 
\]
\qed
\end{prop}

Note that the perturbed
differential $d_\Lamed$ can be written for any system $\Lamed=\|\lambda_{ij}\|_{i\neq j}$
of complex numbers.  Conceptually, $\Lamed$ is a $\CC$-valued function on the
root system of type $A_{n-1}$. We simply refer to $\Lamed$ as a matrix.

By a {\em wheel} for $\Lamed$ we mean a sequence of
$i_1, ..., i_m$ of indices such that
\[
\lambda_{i_1, i_2} \,=\,\lambda_{i_2, i_3}\,=\cdots =\,\lambda_{i_{p-1}, i_p}\,=\,\lambda_{i_p, i_1}
\,=\, 0. 
\]
Theorem \ref{thm:bar} is now a consequence of the following result.

\begin{prop}\label{prop:lamed-without-wheels}
Let  $\Lamed=\|\lambda_{ij}\|_{i\neq j}$ be an $n$ by $n$ matrix without wheels.
Then 
$\bigl(C^\bullet(P_n), d_\Lamed\bigr)$ is exact outside of the leftmost term, where
the cohomology (kernel) is 1-dimensional. 
\end{prop}

 \noindent {\sl Proof:} 
For a face $F=[I_1, ..., I_p]$ of $P_n$ we put
\be\label{eq:lambda-F}
\lambda_F \,\,=\,\,\prod_{\mu < \nu} \prod_{i\in I_\mu\atop j\in I_\nu} \lambda_{ij}.
\ee
Then for an embedding $F\subset F'$ of codimension 1 we have
\[
\lambda_F\,\,=\,\,\lambda_{F'}\cdot \lambda_{FF'}.
\]
This means that we have a morphism of complexes
\[
\Psi: \bigl( C^\bullet(P), d_\Lamed\bigr) \lra
 \bigl( C^\bullet(P), d\bigr),\quad \Psi(\1_F)= \lambda_F \cdot \1_F,
\]
where $d$ is the usual cochain differential. As $P_n$ is a convex polytope,
$ \bigl( C^\bullet(P), d\bigr)$ is exact outside the leftmost term, with $H^0=\CC$. 
We now analyze the kernel and cokernel of $\Psi$. For a face $F$ of $P_n$ as before
we call the {\em depth} of $P$ the number of factors in \eqref{eq:lambda-F}
which are zero. In other words, we put
\be\label{eq:subset-Z}
Z \,\,=\,\,\{ (i,j): \,\, i\neq j, \,\,\lambda_{ij}=0\} \,\,\subset \,\,\{1, ..., n\}^2.
\ee
Then the depth of $F$ is the number
\be\label{eq:def-dep}
\dep(F) \,\,=\,\,\# \bigl\{ (i,j)\in Z: \,\, \exists \,\mu <\nu: \,\, i\in I_\mu, j\in I_\nu\bigr\}.
\ee
Note that if $F$ is a subface of $F'$, then $\dep(F)\geqslant\dep(F')$. 
Therefore we have  a descending chain of  polyhedral subcomplexes
\[
P^{(r)} \,\,=\,\,\bigcup_{\dep(F)\geqslant r} F \,\,\,\subset \,\,\, P_n, \quad r\geqslant 0. 
\]

\begin{lemma}
(a) The complex $Coker(\Psi)$ is isomorphic to the relative cochain complex $C^\bullet(P_n, P^{(1)})$.

(b) The complex $\Ker(\Psi)$ has a filtration with quotients isomorphic to the relative 
cochain complexes $C^\bullet(P^{(r)}, P^{(r+1)})$, $r\geqslant 1$. 
\end{lemma}
 
 \noindent {\sl Proof:} 
 The matrix of $\Psi$ is diagonal in the chosen bases, and $\Im(\Psi)\subset C^\bullet(P_n)$
 is spanned by the $\1_F$, $F\in P^{(1)}$, which shows (a). As for (b), for each $r\geqslant 0$
 we have the cochain subcomplex
  $
 C^\bullet(P_n)^{\geqslant r} \subset C^\bullet(P)
 $
 spanned by $\1_F$ with $\dep(F)\geqslant r$, with $C^\bullet(P_n)^{\geqslant 1}=\Ker(\Psi)$.
 The quotient $C^\bullet(P_n)^{\geqslant r}/C^\bullet(P_n)^{\geqslant r+1}$ is identified with
  $C^\bullet(P^{(r)}, P^{(r+1)})$ in a way similar to (a). \qed
  
  Note that the weights of faces of $P_n$ and
  the polyhedral subcomplexes $P^{(r)}$ are defined entirely in terms of  the subset
  $Z$ in \eqref{eq:subset-Z} which can be, a priori,  arbitrary. 
  Now, absense of wheels in $\Lamed$ (or, what is the same, in $Z$)
  means that after  an appropriate renumbering of $\{1, ..., n\}$, 
 any $(i,j)\in Z$ satisfies $i<j$.  Such renumbering does not
 change the combinatorial type of any of the $P^{(r)}$. 
  Proposition 
  \ref{prop:lamed-without-wheels} is therefore a consequence of the
  following purely combinatorial fact.
  
  \begin{prop}
  Let $Z\subset  \{(i,j)| \, 1\leqslant i< j\leqslant n\}$ be any subset  
  of positive roots for $A_{n-1}$.
    Then each polyhedral complex $P^{(r)}$ is  either empty or contractible. 
   \end{prop}
   \begin{proof}
   For a permutation  $\sigma\in\Sen_n$ let
   \[
   O(\sigma) \,\,=\,\,\bigl\{ (i<j) |\,\,\sigma(i) < \sigma(j)\bigr\}
   \]
   be the set of order preserving pairs of $\sigma$. Thus the weak Bruhat order
   on $\Sen_n$ is given by
   \[
   \sigma \leqslant \tau \quad \text{iff}\quad O(\tau)\subseteq O(\sigma). 
   \]
   Now, fir a face $F\subset P_n$ we have 
   \be\label{eq:dep=min}
   \dep(F) \,\,\,=\,\,\min_{[\sigma]\in\Vertt (F)}\, |O(\sigma)\cap Z|. 
   \ee
   Indeed, for $F=[\sigma]$ a vertex this is precisely the definition 
   \eqref{eq:def-dep}, while for $F=[I_1, ..., I_p]$ the minimum in the RHS of 
   \eqref{eq:dep=min} is achieved for $\sigma$ arranging each $I_\nu$ in the decreasing
   order and is equal to $\dep(F)$. 
   
   Let $D=|Z|$. Then for $r>d$ we have $P^{(r)}=\emptyset$, while for $r\leqslant d$ we have that $P^{(r)}$
   contains at least the vertex $[e]$ corresponding to the unit permutation. Further, by 
   \eqref{eq:dep=min}, the set $\Vertt(P^{(r)})\subset\Sen_n$ is a ``left order ideal" with respect
   to the weak Bruhat order:  with each $\tau$, it contains all $\sigma\leqslant\tau$. This
   implies that $P^{(r)}$ contracts onto $[e]$. 
   \end{proof}
   This finishes the proof of Theorem \ref{thm:bar}. 
   
   \paragraph {F.  Proof of Theorem \ref{thm:cubic-relations-zeta} (sketch).} (a) It is enough to prove that the map
   of the fibers  $\Bc^{-2}_{2, T}\to \Bc^{-1}_{2, T}$ over any $T=\{s^0_1, s^0_2\}\in\Sym^2(\CC)$ is surjective. 
   If $s^0_1\neq s^0_2$, it follows from Theorem \ref{thm:bar}, as there are no wheels of length 2. 
   Assume now that $s^0_1=s_2^0=s^0$. The fiber of $p_{1,1*}\Oc_{\CC^2}$ at $\{s^0,s^0\}$ is then
$\Oc(\CC)/\men_{s^0}^2$,  the space of
first jets of sections of $\Oc_\CC$ at $s^0$. 
Since $\Lambda(s)$ has a first order pole at $0$ with residue 1, for any analytic function
$f(s_1, s_2)$ we have
\[
\lim_{s_1, s_2\to s^0} (\widehat \star_{1,1}F)(s_1, s_2) \,\, =\,\, {1\over 2} {d\over dt}\biggr|_{t=0}
F(s^0+t, s^0-t).
\]
This implies that the  subspace $\men_{s^0}/\men_{s_0}^2$ of jets vanishing at $s$, will map
surjectively onto the fiber of $\Oc_{\Sym^2(\CC)}$ at $\{s^0,s^0\}$. 

\vskip .2cm

(b) 
For $T=\{s_1, s_2, s_3\} \in\Sym^3(\CC)$ let $\CC_T$ be the skyscraper sheaf at $T$.
 We have a spectral sequence
\be\label{eq:SS-Tor}
E_2^{ij} = \operatorname{Tor}_i^{\Sym^3(\CC)}(\underline {H}^j(\Bc_3^\bullet), \CC_T) \,\, \Longrightarrow
\,\, H^{j-i}(\Bc_{3, T}^\bullet).
\ee
We analyze it backwards,
  using the information about the abutment to say something
about $E_2$ and then about  the $\underline {H}^j(\Bc_3^\bullet)$.
Some parts of this analysis involve straightforward computations
which we omit, highlighting the conceptual points only.

\vskip .2cm

First, let $\Delta\subset\Sym^3(\CC)$ be the locus of $T$ such that $s_i=s_j$ for
some $i\neq j$. Note that $W\cap\Delta=\emptyset$.
Theorem \ref{thm:bar} implies that for $T\notin W\cup\Delta$ the abutment of 
 \eqref{eq:SS-Tor} is zero for $j-i>-3$ and this implies that both $\underline H^{-2}$ and
 $\underline H^{-1}$ of $\Bc^\bullet_3$ are zero outside $ W\cup\Delta$. 

\vskip .2cm
 
Next,  $\Bc_3^{-1}=\Oc_{\Sym^3(\CC)}$, so $d_{-2}(\Bc_3^{-1})$ is a sheaf of ideals there and therefore
$\underline{H}^{-1}(\Bc_3^\bullet)$ is the structure sheaf of an analytic subspace $\Wc\subset \Sym^3(\CC)$.
By the above the support of $\Wc$ is contained in $W\cup\Delta$. 

\vskip .2cm

Next, we analyze  \eqref{eq:SS-Tor} in the case when $T\in W$. 
The permutohedron $P_3$ is a hexagon, so for $T\notin\Delta$
the complex $\Bc^\bullet_{3,T}$ is, by Proposition  \ref{prop:B-to-C},
the perturbed cochain complex of this hexagon corresponding to the matrix
$\Lamed = \|\lambda_{ij}\| =\|\Lambda(s_i-s_j)\|$. If $T\in W$, then, after renumbering,
we have $\lambda_{12}=\lambda_{23}=\lambda_{31}=0$, while other $\lambda_{ij}\neq 0$. 
>From this it is an elementary computation to find the dimensions of the cohomology 
spaces of
$\Bc^\bullet_{3,T}$ to be
\[
h^{-3}=3, \,\, h^{-2}=3, \,\, h^{-1}=1.
\]
This, shows that $\Wc$  contains $W$. Further, let $\rho$ be a nontnrivial zero of $\zeta(s)$
of multiplicity $\nu$ and $\Wc_\rho$ be the part of $\Wc$
supported on $W_\rho$.
We can then analyze the last map in complex $\Bc^\bullet_3$
near  $T=\{\rho+c, 1-\rho+c, -1+c\}\in W_\rho$ directly, 
 using the family of perturbed differentials $d_\Lamed: C^1(P_3)\to C^2(P_3)$ with $\Lamed=\|\Lambda(s_i-s_j)\|$
 depending on $\{s_1, s_2, s_3\}$ near $W_\rho$. This is again an elementary computation
 which yields that
  $\Wc_\rho$ is
isomorphic to the $\nu$th infinitesimal neighborhood of $W_\rho$
in an embedded surface. In particular, if $\rho$ is a simple root, then $\Wc_\rho=W_\rho$
as an analytic subspace.

This means that $\Wc=\underline{H}^{-1}(\Bc^\bullet_3)$ is given locally in $\Sym^3(\CC)-\Delta$ by two equations and so
$\dim \operatorname{Tor}_1^{\Sym^3(\CC)}(\underline{H}^{-1}(\Bc_3^\bullet), \CC_T)=2$
for any $T\in W$. From the equality $h^{-2}(\Bc^\bullet_{3,T})=3$ 
and the spectral sequence \eqref{eq:SS-Tor} we then conclude
that $\dim(\underline H^{-2}(\Bc_3^\bullet)\otimes\CC_T)=1$, and so $W\subset \supp(\underline H^{-2}(\Bc^\bullet_3))$.  
The statement that $\underline H^{-2}(\Bc^\bullet_3)=\Oc_{W_\rho}$ near $W_\rho$ for a simple root $\rho$,
uses an additional local calculation which we omit. We also omit the analysis of the  case
$T\in\Delta$ which shows that the support of $\underline H^{-2}(\Bc^\bullet_3)$ 
does not meet $\Delta$. 
\qed

 \footnotesize
  {
  Authors' addresses:

\begin{itemize}
\item[] M.K.: Department of Mathematics, Yale University, 10 Hillhouse Avenue, New Haven CT 06520 USA, 
email: {\tt mikhail.kapranov@yale.edu}

\item[] O. S.: D\'epartement de Math\'ematiques,
B\^atiment 425, 
Facult\'e des Sciences d'Orsay, 
Universit\'e Paris-Sud 11
F-91405 Orsay Cedex, France, email: {\tt  olivier.schiffmann@gmail.com}

\item[] E. V.:  Institut de Math\'ematiques de Jussieu, UMR 7586,  Universit\'e Paris-7 Denis Diderot, 
UFR de Math\'ematiques, Case 7012, 75205 Paris Cedex 13, France, email:
{\tt vasserot@math.jussieu.fr}
\end{itemize}
}

\end{document}